\documentclass{article}

\usepackage{amsmath}
\usepackage{amssymb}
\usepackage{amsthm}

\theoremstyle{plain}
\newtheorem{theorem}{Theorem}
\newtheorem{lemma}[theorem]{Lemma}
\newtheorem{proposition}[theorem]{Proposition}
\newtheorem{corollary}[theorem]{Corollary}

\theoremstyle{definition}
\newtheorem{definition}[theorem]{Definition}
\newtheorem{example}[theorem]{Example}

\theoremstyle{remark}
\newtheorem{remark}[theorem]{Remark}

\begin{document}
\title{Structure of the coadjoint orbits of Lie groups
\thanks{Partially supported by the Ministry of Science and
Innovation, Spain, under Project MTM2008--01386.}
}

\author{Ihor V. Mykytyuk\\
\small Institute of Applied Problems of Mathematics and Mechanics, \\
\small Naukova Str. 3b, 79601, Lviv, Ukraine, mykytyuk\_i@yahoo.com
}

\maketitle

\begin{abstract}
We study the geometrical structure of the coadjoint orbits
of an arbitrary complex or real Lie algebra
${\mathfrak g}$ containing some ideal
${\mathfrak n}$. It is shown that any coadjoint orbit in
${\mathfrak g}^*$ is a bundle with the affine subspace of
${\mathfrak g}^*$ as its fibre. This fibre is an isotropic
submanifold of the orbit and is defined only by the
coadjoint representations of the Lie algebras
${\mathfrak g}$ and ${\mathfrak n}$ on the dual space
${\mathfrak n}^*$. The use of this fact and an application
of methods of symplectic geometry give a new insight into
the structure of coadjoint orbits and allow us to generalize
results derived earlier in the case when
${\mathfrak g}$ is a split extension using the Abelian ideal
${\mathfrak n}$ (a semidirect product). As applications, a
new proof of the formula for the index of Lie algebra and
a necessary condition of integrality of a coadjoint orbit are
obtained.
\end{abstract}

\section{Introduction}

A Lie algebra is a semidirect product if it is a split
extension using its Abelian ideal. The structure of the
coadjoint orbits of a semidirect product is well understood
and known due to papers of Rawnsley~\cite{Ra},
Baguis~\cite{Ba}, Panyushev~\cite{Pa07-1,Pa07-2,Pa03}
and others~\cite{Rais,RT,TY,Ya}. According to~\cite{Ra}, the
coadjoint orbits of a semidirect product are classified by
the coadjoint orbits of so-called little-groups
(reduced-groups) which are isotropy subgroups of some
representations. In fact, the fibre bundles having these
coadjoint orbits as fibres, completely characterize the
coadjoint orbits of the semidirect product. Our paper is
devoted to a generalization of these results of Rawnsley for
arbitrary Lie algebras. While in~\cite{Ra} and~~\cite{Ba}
for calculations the exact multiplication formulas
were used, our approach in the general case is completely
different.

Let $G$ be a connected Lie group with a normal connected
subgroup $N$ and let ${\mathfrak g}$ and
${\mathfrak n}$ be their Lie algebras. Since
${\mathfrak n}$ is an ideal of
${\mathfrak g}$, the coadjoint action of $G$ on
${\mathfrak g}^*$ induces the $G$-action
$\cdot$ on ${\mathfrak n}^*$. Our considerations in the
article are based on the following two facts:
\medskip

\parbox{.85\textwidth}{
An arbitrary coadjoint orbit
${\mathcal O}$ in
${\mathfrak g}^*$ is a bundle with some affine subspace
${\mathcal A}\subset{\mathcal O}\subset{\mathfrak g}^*$ of dimension
$\dim{\mathcal A}=\dim(G\cdot\nu)-\dim(N\cdot\nu)$
as its fibre, where
$\nu=\sigma|{\mathfrak n}\in{\mathfrak n}^*$ and
$\sigma\in{\mathcal A}$. The affine subspace
${\mathcal A}$ is an isotropic submanifold of the orbit
${\mathcal O}$ with respect to the canonical
Kirillov-Kostant-Souriau symplectic structure on
${\mathcal O}$.
}\hfill(*)

\medskip

\parbox{.85\textwidth}{
The identity component of the isotropy group
$N_\nu=\{n\in N: n\cdot\nu=\nu\}$ of
$\nu$ acts transitively on the affine subspace
${\mathcal A}\subset{\mathcal O}$.
}\hfill(**) \medskip

The fact (**) is equivalent to the so-called ``Stages
Hypothesis'', which is a sufficient condition for a general
reduction by stages theorem and was formulated in the paper
of Marsden et al.~\cite{MMPR}. In their
monograph~\cite{MMOPR} this hypothesis was verified for all
split extensions ${\mathfrak g}$ using the Lie algebra
${\mathfrak n}$. Reformulating ``Stages Hypothesis'' in the
form (**) we found a short Lie-algebraic proof of this
hypothesis for all Lie algebras
${\mathfrak g}$ in our paper~\cite{MS}. A slight
modification of this proof and the using of Rawnsley's
approach~\cite{Ra} allow us to prove the fact (*) in this
paper (Theorem~\ref{th.18}) and to generalize results
derived earlier in the case of semidirect products by
Rawnsley~\cite{Ra}. In this direction our aim is to give, on
one hand, a description of the geometrical structure of
the coadjoint orbits in terms of fibre bundles having little
(reduced) algebra coadjoint orbits as fibres
(Proposition~\ref{pr.17}). On the other hand, we investigate
in detail the structure of the isotropy subgroups with
respect to the coadjoint representation of the Lie algebra
${\mathfrak g}$ and the little (reduced) algebra
(Proposition~\ref{pr.26}) and apply this to formulate
necessary conditions for the integrality of the coadjoint orbit of
${\mathfrak g}$ (Proposition~\ref{pr.27}). Proving the
non-sufficiencies of this condition even in the semidirect
product case, we show that the assertion~\cite[Corollary to
Prop.2]{Ra} is not correct (see Remark~\ref{re.30}).

The index of a Lie algebra is defined as the codimension in
the dual space of a coadjoint orbit of the maximal
dimension. The description of the geometrical structure of
the coadjoint orbits mentioned above gives us a new proof of
the formula for the index of Lie algebra
(Theorem~\ref{th.13} and Corollary~\ref{co.15}) obtained by
Panasyuk for arbitrary Lie algebras~\cite{Pa}. Moreover, our
approach allows us to find the direct connection of this
formula with the geometrical structure of the coadjoint
orbits of the little (reduced) Lie algebra. Remark that the proof
in~\cite{Pa} is based on the so-called ``symplectic
reduction by stages'' scheme~\cite{MMOPR} and on
calculations of the ranks of some Poisson submanifold of
${\mathfrak g}^*$ by the construction of the dual pairs of
Poisson manifolds. This formula for index is a
generalization of the well-known Ra\"{\i}s' formula for
semidirect products~\cite{Rais}. As another generalization
of the Ra\"{\i}s' formula we can mention Panyushev's index
formula~\cite{Pa05} for some subclass of graded Lie algebras.
Remark also that the index of representations associated with
stabilizers and so-called representations with good index behavior
was considered by Panyushev and Yakimova in the paper~\cite{Pa06}.

Summarizing the results of this
article, we mention the following point:
\begin{itemize}
\item
the properties (*) and (**) guarantee the existence of some
natural linear structure on the space of
$N_\nu$-orbits and, consequently, the interpretation of this
space as the dual space of some (reduced) Lie algebra, the
interpretation of the orbits in this space as the coadjoint
orbits in this dual space.
\end{itemize}

\section{Coadjoint orbits and their affine subspaces
defined by the ideal}

\subsection{Definitions and notation}

Let ${\mathfrak g}$ be a Lie algebra over the ground field
${\mathbb F}$, where ${\mathbb F}={\mathbb R}$ or
${\mathbb C}$, and
$\rho:{\mathfrak g}\to\operatorname{End}(V)$
be its finite-dimensional representation. Denote by
$\rho^*:{\mathfrak g}\to\operatorname{End}(V^*)$
the dual representation of
${\mathfrak g}$. An element $w\in V^*$ is called
${\mathfrak g}$-{\it regular} whenever its isotropy algebra
${\mathfrak g}_w=\{\xi\in{\mathfrak g}:\rho^*(\xi)w=0\}$
has minimal dimension. The set of all
${\mathfrak g}$-regular elements is open and dense in
$V^*$. Moreover, this set is Zariski open in $V^*$.
The non-negative integer
$\dim{\mathfrak g}_w$, where $w\in V^*$ is
${\mathfrak g}$-regular, is called the {\it index} of the
representation $\rho$ and is denoted by
$\operatorname{ind}({\mathfrak g},V)$. The {\it index}
$\operatorname{ind}{\mathfrak g}$ of the Lie algebra
${\mathfrak g}$ is the index of its coadjoint
representation, or equivalently, the dimension of the
isotropy algebra of a
${\mathfrak g}$-regular element in the dual
${\mathfrak g}^*$ (with respect to the coadjoint
representation). The set of all ${\mathfrak g}$-regular elements
in ${\mathfrak g}^*$ is denoted by $R({\mathfrak g}^*)$.

For any subspace ${\mathfrak a}\subset{\mathfrak g}$ (resp.
$V\subset{\mathfrak g}^*$) denote by
${\mathfrak a}^\bot\subset{\mathfrak g}^*$ (resp.
$V^\bot$) its annihilator in ${\mathfrak g}^*$ (resp. in
${\mathfrak g}$). It is clear that
$({\mathfrak a}^\bot)^\bot={\mathfrak a}$. A subset
${\mathcal A}\subset{\mathfrak g}^*$
will be called an affine
$k$-subspace if it is of the form
${\mathcal A}=\sigma+V$ where
$\sigma\in{\mathfrak g}^*$ is an element and
$V\subset{\mathfrak g}^*$ is a subspace of dimension
$k$. The direct and semi-direct products of Lie algebras are
denoted by
$\times$ and $\ltimes$ respectively. The direct sums of
spaces are denoted by
$\dotplus$. The identity component of an arbitrary Lie group
$H$ is denoted by $H^0$. We will write
$\pi_j$ for the $j$-homotopy group of a manifold. Also we
will often use the following well known statement on the
topology of homogeneous spaces (see~\cite[Ch.III, \S
6.6]{Bo} and~\cite[Ch.1,\S 3.4]{OV}):
\begin{lemma}\label{le.1}
For a connected Lie group
$K$ and its (not necessary closed) subgroup
$H$ the following holds: $1)$ if
$H$ is a normal subgroup of $K$ and
$\pi_1(K)=0$ then $H=\overline{H}$ and
$\pi_1(H^0)=\pi_1(K/H^0)=0$; $2)$ if
$H=\overline{H}$, $\pi_1(K/H)=\pi_2(K/H)=0$ then
the Lie subgroup $H$ is connected, i.e. $|H/H^0|=1$,
and $\pi_1(K)\simeq\pi_1(H)$; $3)$ if
$H=\overline{H}$ and $\pi_1(K)=0$ then
$\pi_1(K/H)\simeq H/H^0$. Here
$\overline{H}$ denotes the closure of $H$ in $K$.
\end{lemma}

\subsection{Coadjoint orbits and their isotropy groups}

Let $G$ be a connected real or complex Lie group with a normal
connected subgroup
$N\subset G$ (not necessary closed). Denote by
${\mathfrak g}$ and ${\mathfrak n}$
the corresponding Lie algebras. Since the Lie group
$N$ is a normal subgroup of $G$, we have
\begin{equation}\label{eq.1}
\operatorname{Ad}_n\xi-\xi\in{\mathfrak n}
\quad\text{for all}\quad
n\in N,\ \xi\in{\mathfrak g}.
\end{equation}
This fact is well known if the subgroup
$N$ is closed. To prove~\eqref{eq.1} in our general case it
is sufficient to remark that the curve
$n(\exp(t\xi)n^{-1}\exp(-t\xi))$ is the curve in
$N$ passing through the identity element.

Let $\operatorname{Ad}^*:G\to\operatorname{End}({\mathfrak g}^*)$
be the coadjoint representation of the Lie group
$G$ on the dual space ${\mathfrak g}^*$.
Since we shall consider also some subgroups of
$G$, by $\operatorname{Ad}^*_g$ and
$\operatorname{ad}^*_\xi$ we shall denote only the operators
on the space ${\mathfrak g}^*$, by $\operatorname{Ad}_g$ and
$\operatorname{ad}_\xi$ the operators on the Lie algebra
${\mathfrak g}$. Fix some linear functional
$\sigma\in{\mathfrak g}^*$. Denote by
$G_\sigma$ the isotropy group of
$\sigma$ (with respect to the coadjoint representation of
$G$) and by ${\mathfrak g}_\sigma$ its Lie algebra. Put
$N_\sigma=N\cap G_\sigma$ and
${\mathfrak n}_\sigma={\mathfrak n} \cap{\mathfrak g}_\sigma$.
The subgroup $N_\sigma$ is a closed subgroup in
$N$ with the Lie algebra
${\mathfrak n}_\sigma$. By the definition,
\begin{equation}\label{eq.2}
{\mathfrak g}_\sigma
=\{\xi\in{\mathfrak g}: \langle\sigma,[\xi,{\mathfrak g}]\rangle=0\}
\qquad\mbox{and}\qquad
{\mathfrak n}_\sigma
=\{y\in{\mathfrak n}: \langle\sigma,[y,{\mathfrak g}]\rangle=0\}.
\end{equation}

Since the subalgebra ${\mathfrak n}$ is an ideal of
${\mathfrak g}$, the adjoint representations of
${\mathfrak g}$ induce the representation
$\rho$ of ${\mathfrak g}$ in
${\mathfrak n}$, the adjoint action
$\operatorname{Ad}:G\to\operatorname{End}({\mathfrak g})$ of
$G$ induces $G$-action on ${\mathfrak n}$:
$G\times{\mathfrak n}\to{\mathfrak n}$,
$(g,y)\mapsto\operatorname{Ad}_g y$.
For the dual representation
$\rho^*$ of ${\mathfrak g}$ in
${\mathfrak n}^*$ we have:
$$
\langle\rho^*_\xi\mu,y\rangle=\langle\mu,[\xi,y]\rangle,
\quad\text{where}\
\xi\in{\mathfrak g}, \mu\in{\mathfrak n}^*, y\in{\mathfrak n}.
$$
The corresponding $G$-action on
${\mathfrak n}^*$ is defined by the equation
$\langle g\cdot\mu,y\rangle=\langle\mu,\operatorname{Ad}_g y\rangle$.
The restriction of this action on the subgroup
$N\subset G$ is its coadjoint action. Moreover, the
canonical projection
${\Pi_1^{\mathfrak g}}:{\mathfrak g}^*\to{\mathfrak n}^*$,
$\beta\mapsto\beta|{\mathfrak n}$ is a
$G$-equivariant mapping with respect to these two actions of
$G$ on the spaces ${\mathfrak g}^*$ and
${\mathfrak n}^*$ respectively:
$$
{\Pi_1^{\mathfrak g}}(\operatorname{Ad}_g^*\beta)
=g\cdot{\Pi_1^{\mathfrak g}}(\beta),
\qquad
\mbox{for all}
\ \beta\in{\mathfrak g}^*, g\in G.
$$
Indeed, for any $y\in{\mathfrak n}$
\begin{equation}\label{eq.3}
\langle{\Pi_1^{\mathfrak g}}(\operatorname{Ad}_g^*\beta),y\rangle=
\langle\operatorname{Ad}_g^*\beta,y\rangle=
\langle\beta,\operatorname{Ad}_g y\rangle=
\langle{\Pi_1^{\mathfrak g}}(\beta),\operatorname{Ad}_g y\rangle=
\langle g\cdot{\Pi_1^{\mathfrak g}}(\beta),y\rangle.
\end{equation}

On the other hand, the canonical homomorphism
$\pi:{\mathfrak g}\to{\mathfrak g}/{\mathfrak n}$
induces the canonical linear embedding
$\pi^*:({\mathfrak g}/{\mathfrak n})^*\to{\mathfrak g}^*$.
The following lemma is known. We will prove it for
completeness and also because the proof will be used to give
a more general result.
\begin{lemma}\label{le.2}
The canonical linear embedding
$\pi^*:({\mathfrak g}/{\mathfrak n})^*\to{\mathfrak g}^*$
maps each coadjoint orbit
${\mathcal O}_{\mathfrak b}$ of the quotient Lie algebra
${\mathfrak b}={\mathfrak g}/{\mathfrak n}$ onto some coadjoint orbit
${\mathcal O}_{\mathfrak g}$ of ${\mathfrak g}$.
This map defines a one-to-one correspondence between the
set of all coadjoint orbits in
$({\mathfrak g}/{\mathfrak n})^*$
and the set of all coadjoint orbits in
${\mathfrak g}^*$ belonging to the annihilator
${\mathfrak n}^\bot\subset{\mathfrak g}^*$.
Moreover, the restriction
$\pi^*: {\mathcal O}_{\mathfrak b}\to {\mathcal O}_{\mathfrak g}$
of the map $\pi^*$ is a symplectic map, i.e.
$(\pi^*|{\mathcal O}_{\mathfrak b})^*(\omega_{\mathfrak
g})=\omega_{\mathfrak b}$,
where $\omega_{\mathfrak g}$ and $\omega_{\mathfrak b}$
are the canonical Kirillov-Kostant-Souriau symplectic
2-forms on the coadjoint orbits
${\mathcal O}_{\mathfrak g}\subset{\mathfrak g}^*$ and
${\mathcal O}_{\mathfrak b}\subset{\mathfrak b}^*$
respectively.
\end{lemma}

\begin{proof}
Since ${\mathfrak n}$ is an ideal of
${\mathfrak g}$, there exists a unique homomorphism
$\varphi$ of the Lie group
$G$ into the group of all automorphisms of the Lie algebra
${\mathfrak b}={\mathfrak g}/{\mathfrak n}$ such that
$\varphi(g)\circ\pi=\pi\circ\operatorname{Ad}_g$,
$g\in G$. The connected Lie group
$\operatorname{Ad}(G)$ is the group of inner automorphisms
of ${\mathfrak g}$. Since each inner derivation of
${\mathfrak b}$ is induced by some inner derivation of
${\mathfrak g}$, the image $B=\varphi(G)$ of $G$
is the Lie group of inner automorphisms of the Lie algebra
${\mathfrak b}$. Taking into account that
$\operatorname{Ad}^*_g\circ\pi^*=\pi^*\circ(\varphi(g))^*$,
$\pi^*({\mathfrak b}^*)={\mathfrak n}^\bot$ and
$\operatorname{Ad}^*(G)({\mathfrak n}^\bot)={\mathfrak n}^\bot$, we compete the proof of the
first assertion.

Choose some element
$\beta'\in{\mathcal O}_{\mathfrak b}$ and put $\beta=\pi^*(\beta')$.
To prove the second assertion, remark that the map
$\pi^*$ is linear. Then
$d(\pi^*)(\beta')=\pi^*$ and for any $\xi,\eta\in{\mathfrak g}$
\begin{equation*}
\begin{split}
\langle d(\pi^*)(\beta')(\widetilde{\operatorname{ad}}^*_{\pi(\xi)}\beta'),\eta\rangle
&=\langle\widetilde{\operatorname{ad}}^*_{\pi(\xi)}\beta',\pi(\eta)\rangle
=\langle\beta',[\pi(\xi),\pi(\eta)]_{\mathfrak b}\rangle \\
&=\langle\beta',\pi([\xi,\eta])\rangle
=\langle\operatorname{ad}^*_\xi\beta,\eta\rangle.
\end{split}
\end{equation*}
Thus $((\pi^*)^*\omega_{\mathfrak g})(\beta')(\widetilde{\operatorname{ad}}^*_{\pi(\xi)}\beta',
\widetilde{\operatorname{ad}}^*_{\pi(\eta)}\beta')
=\omega_{\mathfrak g}(\beta)(\operatorname{ad}^*_\xi\beta,\operatorname{ad}^*_\eta\beta)$
and by definition of $\omega_{\mathfrak g}$
$$
\omega_{\mathfrak g}(\beta)(\operatorname{ad}^*_\xi\beta,\operatorname{ad}^*_\eta\beta)
\stackrel{\mathrm{def}}{=}\langle\beta,[\xi,\eta]\rangle
=\langle\beta',[\pi(\xi),\pi(\eta)]_{\mathfrak b}\rangle
\stackrel{\mathrm{def}}{=}\omega_{\mathfrak b}(\beta')(\widetilde{\operatorname{ad}}^*_{\pi(\xi)}\beta',
\widetilde{\operatorname{ad}}^*_{\pi(\eta)}\beta').
$$
i.e. $(\pi^*|{\mathcal O}_{\mathfrak b})^*(\omega_{\mathfrak g})=\omega_{\mathfrak b}$.
\end{proof}

\begin{remark}\label{re.3}
As follows from Lemma~\ref{le.2} the set of all coadjoint
orbits of the Lie algebra ${\mathfrak g}$ contains the coadjoint orbits
of all its quotient algebras.
\end{remark}

For the element
$\sigma\in{\mathfrak m}^*$, consider its restriction
$\nu=\sigma|{\mathfrak n}\in{\mathfrak n}^*$. Denote by
$G_\nu$ and $N_\nu$ the isotropy groups of the element
$\nu$ with respect to the $\rho^*$-action, by
${\mathfrak g}_\nu$ and ${\mathfrak n}_\nu$
the corresponding Lie algebras. It is clear that
${\mathfrak n}_\nu={\mathfrak n}\cap{\mathfrak g}_\nu$
and the subgroup $N_\nu=N \cap G_\nu$ is a normal subgroup of
$G_\nu$. Remark here, that $N_\nu$
is also the usual isotropy group for coadjoint representation
of the Lie group $N$ on the dual space ${\mathfrak n}^*$.

Since $[{\mathfrak g} ,{\mathfrak n}]\subset{\mathfrak n}$,
by the definition,
\begin{align}\label{eq.4}
{\mathfrak g}_\nu
&=\{\xi\in{\mathfrak g} : \rho^*_\xi\nu=0\}
=\{\xi\in{\mathfrak g} : \langle\nu,[\xi,{\mathfrak n}]\rangle=0\}
=\{\xi\in{\mathfrak g} : \langle\sigma,[\xi,{\mathfrak n}]\rangle=0\},
\\\label{eq.5}
{\mathfrak n}_\nu
&=\{y\in{\mathfrak n} : \langle\nu,[y,{\mathfrak n} ]\rangle=0\}
=\{y\in{\mathfrak n} : \langle\sigma,[y,{\mathfrak n} ]\rangle=0\},
\end{align}
and
\begin{equation}\label{eq.6}
G_\nu=\{g\in G: g\cdot \nu=\nu\}=
\{g\in G: \operatorname{Ad}^*_g\sigma|{\mathfrak n}
=\sigma|{\mathfrak n}=\nu\}.
\end{equation}
Note that
\begin{equation}\label{eq.7}
\operatorname{Ad}(G_\nu)({\mathfrak n}_\nu)={\mathfrak n}_\nu
\end{equation}
because
$\operatorname{Ad}(G_\nu)({\mathfrak g}_\nu)={\mathfrak g}_\nu$
and $\operatorname{Ad}(G)({\mathfrak n})={\mathfrak n}$
(by definition~\eqref{eq.5}
$\langle\nu,[{\mathfrak n}_\nu,{\mathfrak n}]\rangle=0$).
Also by the identity
$\operatorname{Ad}^*(G)({\mathfrak n}^\bot)={\mathfrak n}^\bot$,
\begin{align}\label{eq.8}
G_\nu&=\{g\in G: \operatorname{Ad}^*_g({\mathcal A}_\nu)
={\mathcal A}_\nu\},
\text{ where }\\\label{eq.9}
{\mathcal A}_\nu&\stackrel{\mathrm{def}}{=}\sigma+{\mathfrak n}^\bot=
\{\alpha\in{\mathfrak g}^*:\alpha|{\mathfrak n}=\nu\}.
\end{align}

Let
${\mathcal O}^\sigma(G)=\{\operatorname{Ad}_g^*\sigma, g\in G\}$
be the coadjoint orbit of the Lie group
$G$ in ${\mathfrak g}^*$ through the point $\sigma$ and let
${\mathcal O}^\nu(G)=G\cdot\nu$ be the corresponding
$G$-orbit in ${\mathfrak n}^*$. Consider also the orbit
${\mathcal O}^\sigma(G_\nu)\subset {\mathcal O}^\sigma(G)$
in ${\mathfrak g}^*$ of the Lie group $G_\nu$.
\begin{lemma}\label{le.4}
The restriction $p_1={\Pi_1^{\mathfrak g}}|{\mathcal O}^\sigma(G)$
of the natural projection
${\Pi_1^{\mathfrak g}}:{\mathfrak g}^*\to{\mathfrak n}^*$
is a $G$-equivariant submersion of the coadjoint orbit
${\mathcal O}^\sigma(G)$ onto the orbit
${\mathcal O}^\nu(G)$. This map
${\Pi_1^{\mathfrak g}}:{\mathcal O}^\sigma(G)\to{\mathcal O}^\nu(G)$
is a bundle with the total space
${\mathcal O}^\sigma(G)$, the base
${\mathcal O}^\nu(G)$ and the fibre
${\mathcal O}^\sigma(G_\nu)$.
\end{lemma}
To prove the lemma it is sufficient to remark that
${\mathcal O}^\sigma(G)\simeq G/G_\sigma$,
${\mathcal O}^\nu(G)\simeq G/G_\nu$ and
${\mathcal O}^\sigma(G_\nu)\simeq G_\nu/G_\sigma$.

The coadjoint orbit
${\mathcal O}^\sigma(G)\subset{\mathfrak g}^*$
is a symplectic manifold with the symplectic
Kirillov-Kostant-Souriau 2-form
$\omega$:
\begin{equation}\label{eq.10}
\omega(\sigma)(\operatorname{ad}^*_\xi\sigma,
\operatorname{ad}^*_\eta\sigma)\stackrel{\mathrm{def}}{=}
\langle\sigma,[\xi,\eta]\rangle,
\quad\text{where}\quad
\xi,\eta\in{\mathfrak g}.
\end{equation}
Here the tangent space
$T_\sigma{\mathcal O}^\sigma(G)$
is identified, as usual, with the subspace
$\operatorname{ad}^*_{\mathfrak g}\sigma$ of
${\mathfrak g}^*$.
We will say that a submanifold
$M\subset{\mathcal O}^\sigma(G)$
is an {\it isotropic submanifold} of the orbit
${\mathcal O}^\sigma(G)$ if for each point
$\alpha\in M$ the tangent space
$T_\alpha M$ is an isotropic subspace of
$T_\alpha{\mathcal O}^\sigma(G)$ with respect to the form
$\omega$, i.e.
$\omega(\alpha)(T_\alpha M,T_\alpha M)=0$.

Let us consider two orbits
${\mathcal O}^\sigma(N)$ and
${\mathcal O}^\sigma(G_\nu)$ (submanifolds of
${\mathcal O}^\sigma(G)$) through the point
$\sigma$. It follows immediately from
definition~\eqref{eq.10} that the tangent space
$T_\sigma{\mathcal
O}^\sigma(G_\nu)=\operatorname{ad}^*_{{\mathfrak g}_\nu}\sigma$
is an orthogonal complement to the tangent space
$T_\sigma{\mathcal
O}^\sigma(N)=\operatorname{ad}^*_{{\mathfrak n}}\sigma$
in $T_\sigma{\mathcal O}^\sigma(G)$
with respect to the symplectic form
$\omega$:
\begin{equation}\label{eq.11}
\omega(\sigma)(\operatorname{ad}^*_\xi\sigma,
\operatorname{ad}^*_{\mathfrak n}\sigma)\stackrel{\mathrm{def}}{=}
\langle\sigma,[\xi,{\mathfrak n}]\rangle=0
\quad\Longleftrightarrow\quad
\xi\in{\mathfrak g}_\nu.
\end{equation}
Since the form
$\omega$ is non-degenerate and the isotropy algebra
${\mathfrak g}_\sigma$ is a subalgebra of
${\mathfrak g}_\nu$ (see definition~\eqref{eq.2}), we have
\begin{equation}\label{eq.12}
\dim{\mathfrak g}-\dim{\mathfrak g}_\sigma
=(\dim{\mathfrak n}-\dim{\mathfrak n}_\sigma)+
(\dim{\mathfrak g}_\nu-\dim{\mathfrak g}_\sigma).
\end{equation}
This identity can be easily rewritten in the following form
\begin{equation}\label{eq.13}
\dim{\mathfrak n}_\sigma=\dim{\mathfrak n}
-(\dim{\mathfrak g}-\dim{\mathfrak g}_\nu),
\end{equation}
i.e. the dimension of the Lie algebra
$\dim{\mathfrak n}_\sigma$ depends on its restriction
$\nu=\sigma|{\mathfrak n}$ alone. Moreover, by the
commutation relation
$[{\mathfrak g} ,{\mathfrak n}]\subset{\mathfrak n}$,
the algebra
${\mathfrak n}_\sigma$ also depends only on this restriction
$\nu$:
\begin{equation}\label{eq.14}
{\mathfrak n}_\sigma
\stackrel{\mathrm{def}}{=}\{y\in{\mathfrak n}:
\langle\sigma,[y,{\mathfrak g}]\rangle=0\}
=\{y\in{\mathfrak n}: \langle\nu,[y,{\mathfrak g}]\rangle=0\}.
\end{equation}
This Lie algebra and the corresponding connected Lie subgroup of
$N_\nu$ will be denoted by
${\mathfrak n}_{\nu\nu}$ and $N^0_{\nu\nu}$
respectively. In other words, for each element
$\alpha\in{\mathfrak g}^*$ such that
$\alpha|{\mathfrak n}=\sigma|{\mathfrak n}$:
\begin{equation}\label{eq.15}
{\mathfrak n}_{\alpha}={\mathfrak n}_\sigma={\mathfrak n}_{\nu\nu}
\quad\text{and}\quad
N^0_{\alpha}=N^0_\sigma=N^0_{\nu\nu}.
\end{equation}
In particular,
$N^0_{\nu\nu}$ is a closed subgroup of the Lie groups
$N$ and $N_\nu$. Moreover, this subgroup is the connected
component of the closed subgroup
$N_{\nu\nu}$ of $N_\nu\subset N$, where
\begin{equation}\label{eq.16}
N_{\nu\nu}=\{n\in N: \operatorname{Ad}^*_n(\alpha)=\alpha
\ \text{for all}\ \alpha\in {\mathcal A}_\nu\}=
\bigcap_{\alpha\in {\mathcal A}_\nu} N_\alpha.
\end{equation}

However, we can rewrite identity~\eqref{eq.13} in the following
form:
\begin{equation}\label{eq.17}
\dim{\mathfrak n}_\nu-\dim{\mathfrak n}_\sigma
=\dim{\mathfrak g}-(\dim{\mathfrak n}+\dim{\mathfrak g}_\nu
-\dim{\mathfrak n}_\nu).
\end{equation}
The right-hand side of this identity is the codimension of the subspace
${\mathfrak n}+{\mathfrak g}_\nu$ in
${\mathfrak g}$ because by definition
${\mathfrak n}_\nu={\mathfrak g}_\nu\cap{\mathfrak n}$.
The left-hand side of~\eqref{eq.17} is the dimension of the subspace
$\operatorname{ad}^*_{{\mathfrak n}_\nu}\sigma\subset{\mathfrak g}^*$.
But $\operatorname{ad}^*_{{\mathfrak n}_\nu}\sigma$
is a subspace of $({\mathfrak n}+{\mathfrak g}_\nu)^\bot$ because
$\operatorname{ad}^*_{{\mathfrak n}_\nu}\sigma({\mathfrak n})=0$
by definition~\eqref{eq.5} and
$\operatorname{ad}^*_{{\mathfrak n}_\nu}\sigma({\mathfrak g}_\nu)=0$
by~\eqref{eq.4}. Therefore from~\eqref{eq.17} it follows that
\begin{equation}\label{eq.18}
\dim({\mathfrak n}_\nu/{\mathfrak n}_\sigma)
=\dim({\mathfrak n}+{\mathfrak g}_\nu)^\bot
\quad\text{and, consequently,}\quad
\operatorname{ad}^*_{{\mathfrak n}_\nu}\sigma
=({\mathfrak n}+{\mathfrak g}_\nu)^\bot.
\end{equation}

\begin{remark}\label{re.5}
The subspace
$\operatorname{ad}^*_{{\mathfrak n}_\nu}\sigma\subset{\mathfrak g}^*$
is the tangent space to the orbit
${\mathcal O}^\sigma(N_\nu)={\mathcal
O}^\sigma(G_\nu)\cap{\mathcal O}^\sigma(N)$
of the Lie group $N_\nu$ through the point
$\sigma\in{\mathfrak g}^*$ and, as we shown above, this
space is the null space of the restrictions
$\omega|T_\sigma{\mathcal O}^\sigma(G_\nu)$ and
$\omega|T_\sigma{\mathcal O}^\sigma(N)$.
\end{remark}

Our interest now centers on the two orbits in
${\mathfrak g}^*$ (through the element
$\sigma$) mentioned above: ${\mathcal O}^\sigma(G_\nu)$ and
${\mathcal O}^\sigma(N_\nu)$. First of all, we will show
that
${\mathcal O}^\sigma(N^0_\nu)=\sigma+({\mathfrak
n}+{\mathfrak g}_\nu)^\bot$,
i.e. this orbit is an affine subspace of
${\mathfrak g}^*$. To this end, we consider the kernel
${\mathfrak n} ^\natural_\nu\subset{\mathfrak n}_\nu$
of the restriction $\nu|{\mathfrak n}_\nu$, i.e.
${\mathfrak n}^\natural_\nu=\ker\nu\cap{\mathfrak n}_\nu$.
Remark that ${\mathfrak n}^\natural_\nu={\mathfrak n}_\nu$ or
$\dim({\mathfrak n}_\nu/{\mathfrak n}^\natural_\nu)=1$.
By~\eqref{eq.4}
\begin{equation}\label{eq.19}
[{\mathfrak g}_\nu,{\mathfrak n} ]\subset\ker\nu
\quad\text{and}\quad
[{\mathfrak g}_\nu,{\mathfrak n}_\nu]
\subset({\mathfrak g}_\nu\cap{\mathfrak n})\cap\ker\nu
={\mathfrak n}^\natural_\nu,
\end{equation}
so that the subspace
${\mathfrak n}^\natural_\nu$ is an ideal in
${\mathfrak g}_\nu$. Moreover, since $h\cdot\nu=\nu$,
$\operatorname{Ad}_h({\mathfrak n}_\nu)={\mathfrak n}_\nu$
for all $h\in G_\nu$ (see~\eqref{eq.7}) and, by the definition,
$\langle h\cdot\nu,y\rangle=\langle\nu,\operatorname{Ad}_h y\rangle$
for $y\in{\mathfrak n}$, we have
\begin{equation}\label{eq.20}
\operatorname{Ad}_h({\mathfrak n}^\natural_\nu)={\mathfrak n}^\natural_\nu
\quad\text{for all } h\in G_\nu.
\end{equation}
Let $N_\nu^\mathrm{fin}$ be the subgroup of
$N_\nu$ generated by all elements
$n\in N_\nu$ such that the power
$(\operatorname{Ad}_n)^m\in\operatorname{Ad}(N^0_\nu)$
for some integer
$m\in{\mathbb Z}$. This group is a closed Lie subgroup of
$N_\nu$ because it contains the identity component
$N^0_\nu$ of $N_\nu$. We claim that
\begin{equation}\label{eq.21}
\operatorname{Ad}_n\xi-\xi\in{\mathfrak n}^\natural_\nu
\quad\text{for all } \xi\in{\mathfrak g}_\nu
\text{ and } n\in N_\nu^\mathrm{fin}\supset N^0_\nu.
\end{equation}
Relations~\eqref{eq.21} were established in~\cite[\S
5.2]{MMOPR} in the case when the Lie group
$N_\nu$ is connected. We will prove~\eqref{eq.21} modifying
the method used in~\cite{MMOPR}. To this end consider the
representation $n\mapsto \operatorname{Ad}_n|{\mathfrak g}_\nu$
of the Lie group
$N_\nu\subset G_\nu$. This representation induces the
trivial representation of the identity component
$N_\nu^0\subset N_\nu$ in the quotient algebra
${\mathfrak g}_\nu/{\mathfrak n}^\natural_\nu$ because
$[{\mathfrak n}_\nu,{\mathfrak g}_\nu]\subset{\mathfrak n}^\natural_\nu$
(the corresponding homomorphism of Lie algebras is trivial).
Thus relations~\eqref{eq.21} hold for all
$n\in N^0_\nu$, i.e.
$\langle\nu, \operatorname{Ad}_n\xi-\xi\rangle=0$
for all such $n$.

Since $N_\nu$ is a normal (not necessary closed) subgroup of
$G_\nu$, we have
$\operatorname{Ad}_n\xi-\xi\in{\mathfrak n}_\nu$ for all
$n\in N_\nu$ and $\xi\in{\mathfrak g}_\nu$
(see~\eqref{eq.1}). Now to prove~\eqref{eq.21}
we will show that the mapping
$$
\chi_\xi:N_\nu\to{\mathbb F},\quad
\chi_\xi(n)=\langle\nu, \operatorname{Ad}_n\xi-\xi\rangle,
\quad\xi\in{\mathfrak g}_\nu,
$$
is a homomorphism of the group
$N_\nu$ into the additive group ${\mathbb F}$. Indeed,
for $n_1,n_2\in N_\nu$,
\begin{align*}
\langle\nu,\operatorname{Ad}_{n_1n_2}\xi-\xi\rangle
&=\langle\nu,\operatorname{Ad}_{n_1}(\operatorname{Ad}_{n_2}\xi-\xi)
+(\operatorname{Ad}_{n_1}\xi-\xi)\rangle \\
&=\langle\nu,(\operatorname{Ad}_{n_2}\xi-\xi)
+(\operatorname{Ad}_{n_1}\xi-\xi)\rangle,
\end{align*}
because
$n_1\cdot\nu=\nu$. Now, if
$(\operatorname{Ad}_n)^m\in\operatorname{Ad}(N^0_\nu)$ then
$$
m\chi_\xi(n)=\chi_\xi(n^m)
=\langle\nu, (\operatorname{Ad}_n)^m\xi-\xi\rangle=0.
$$
The proof of~\eqref{eq.21} is completed.

For the element $\sigma\in{\mathfrak g}^*$ denote by
$\tau$ its restriction
$\sigma|{\mathfrak g}_\nu$. Using the pair of covectors
$\nu\in{{\mathfrak n}}^*$ and
$\tau\in{\mathfrak g}_\nu^*$ define the affine subspace
${{\mathcal A}}_{\nu\tau}\subset{{\mathcal A}}_{\nu}
\subset{\mathfrak g}^*$
as follows:
\begin{equation}\label{eq.22}
{\mathcal A} _{\nu\tau}=
\{\alpha\in{\mathfrak g}^*: \alpha|{\mathfrak n}
=\nu,\ \alpha|{\mathfrak g}_\nu=\tau\}=
\sigma+({\mathfrak n}+{\mathfrak g}_\nu)^\bot.
\end{equation}
It is clear that
\begin{equation}\label{eq.23}
\begin{split}
\dim {\mathcal A} _{\nu\tau}
&=\mathrm{codim}\, ({\mathfrak n}+{\mathfrak g}_\nu) \\
&=\dim{\mathfrak g}-(\dim{\mathfrak n}
+\dim{\mathfrak g}_\nu-\dim{\mathfrak n}_\nu) \\
&=\dim(G/G_\nu)-\dim(N/N_\nu) \\
&=\dim(G/N)-\dim(G_\nu/N_\nu).
\end{split}
\end{equation}
We claim that this affine subspace
${{\mathcal A}}_{\nu\tau}\subset{\mathfrak g}^*$
is invariant with respect to the action of the Lie group
$N_\nu^\mathrm{fin}\subset N_\nu$
(containing the identity component
$N^0_\nu$ of $N_\nu$). Indeed, let
$\alpha\in{\mathcal A} _{\nu\tau}$ and
$n\in N^\mathrm{fin}_\nu$. Since
$N^\mathrm{fin}_\nu\subset N_\nu= N\cap G_\nu$ and
${\mathcal A}_{\nu\tau}\subset {\mathcal A}_\nu$,
by~\eqref{eq.8} $\operatorname{Ad}_n^*\alpha|{\mathfrak n}=\nu$.
To prove that $\operatorname{Ad}_n^*\alpha|{\mathfrak g}_\nu=\tau$
remark that by~\eqref{eq.21} for all vectors
$\xi\in{\mathfrak g}_\nu$ we have
$$
\langle\operatorname{Ad}_n^*\alpha-\alpha,\xi\rangle=
\langle\alpha,\operatorname{Ad}_n\xi-\xi\rangle\in
\langle\alpha,{\mathfrak n} ^\natural_\nu\rangle=
\langle\nu,{\mathfrak n} ^\natural_\nu\rangle=0,
$$
i.e.
\begin{equation}\label{eq.24}
\operatorname{Ad}^*_n({\mathcal A} _{\nu\tau})
\subset{\mathcal A} _{\nu\tau}
\quad\text{for all } n\in N_\nu^\mathrm{fin}\supset N^0_\nu.
\end{equation}
As we have shown above, the Lie algebra
${\mathfrak n}_\sigma$ is defined by the restriction
$\sigma|{\mathfrak n}=\nu$ alone, therefore
${\mathfrak n}_\sigma={\mathfrak n}_\alpha$
(see~\eqref{eq.15}). By definition
$N_\alpha\subset N_\nu$ and, consequently,
$N^0_\alpha\subset N^0_\nu$. Taking into
account~\eqref{eq.18}, we obtain that the
$\operatorname{Ad}^*(N^0_\nu)$-orbit in the space
${\mathcal A} _{\nu\tau}$ through the element
$\alpha$ (isomorphic to the quotient space
$N^0_\nu/(N_\alpha\cap N^0_\nu)$) is an open subset of
${\mathcal A} _{\nu\tau}$:
$$
\dim N^0_\nu/(N_\alpha\cap N^0_\nu)=
\dim({\mathfrak n}_\nu/{\mathfrak n}_\alpha)=
\dim({\mathfrak n}_\nu/{\mathfrak n}_\sigma)=
\dim({\mathfrak n} +{\mathfrak g}_\nu)^\bot=
\dim {\mathcal A} _{\nu\tau}.
$$
Since the space
${\mathcal A} _{\nu\tau}$ is connected, this orbit is the
whole space ${\mathcal A} _{\nu\tau}$, i.e.
$\operatorname{Ad} ^*(N^0_\nu)$ acts transitively on
${\mathcal A} _{\nu\tau}$. Since the affine space
${\mathcal A} _{\nu\tau}$ is contractible, by
Lemma~\ref{le.1} the isotropy group
$N_\alpha\cap N^0_\nu$ is connected, that is, it is equal to
$N^0_\alpha$ (the identity component of
$N_\alpha\subset N_\nu$). Similarly, the group
$N_\nu^\mathrm{fin}$ acts transitively on
${\mathcal A} _{\nu\tau}$ and, consequently,
$$
N_\nu^\mathrm{fin}/N^0_\nu
\simeq (N_\nu^\mathrm{fin}\cap N_\alpha)/N^0_\alpha.
$$
Also by Lemma~\ref{le.1},
$$
\pi_1(N^0_\nu)\simeq\pi_1(N^0_\alpha)
$$
because
$\pi_1({\mathcal A} _{\nu\tau})=\pi_2({\mathcal A} _{\nu\tau})=0$.
Thus
\begin{equation}\label{eq.25}
{\mathcal A}_{\nu\tau}
\simeq N^0_\nu/N^0_\sigma=N^0_\nu/N^0_{\nu\nu}
\qquad \text{and}\qquad
{\mathcal A}_{\nu\tau}
\simeq N_\nu^\mathrm{fin}/(N_\sigma\cap N_\nu^\mathrm{fin}).
\end{equation}
Since the action of $N^0_\nu$ on ${\mathcal A}_{\nu\tau}$
is transitive and the isotropy group
$N^0_\alpha=N^0_{\nu\nu}$ is the same for all points
$\alpha\in {\mathcal A}_{\nu\tau}$, the group
$N^0_{\nu\nu}$ is a normal subgroup of $N^0_\nu$.

Consider now the isotropy group
$G_\nu$. The algebra ${\mathfrak g}_\nu$
is its tangent Lie algebra. For the element
$\tau=\sigma|{\mathfrak g}_\nu$ denote by
$G_{\nu\tau}$ the isotropy group of
$\tau\in {\mathfrak g}^*_\nu$ with respect to the natural
co-adjoint action of $G_\nu$ on
${\mathfrak g}^*_\nu$, which we denote by
$\widehat{\operatorname{Ad}}^*$. Let
${\mathcal O}^\tau(G_\nu)\subset{\mathfrak g}^*_\nu$
be the corresponding $\widehat{\operatorname{Ad}}^*$-orbit of
$G_\nu$ passing through the point
$\tau$ (the union of disjoint coadjoint orbits in
${\mathfrak g}^*_\nu$). Then
${\mathcal O}^\tau(G_\nu)\simeq G_\nu/G_{\nu\tau}$.
Taking into account that the
$\widehat{\operatorname{Ad}}$-action of $G_\nu$ on
${\mathfrak g}_\nu$ is determined by the
$\operatorname{Ad}$-action of $G$ on
${\mathfrak g}$, we obtain that
the natural projection
\begin{equation}\label{eq.26}
{\Pi_{2}^{\mathfrak g}}:{\mathfrak g}^*\to{\mathfrak g}_\nu^*,\quad
\beta\mapsto\beta|{\mathfrak g}_\nu,
\end{equation}
is a $G_\nu$-equivariant map with respect to the coadjoint
actions $\operatorname{Ad}^*$ and
$\widehat{\operatorname{Ad}}^*$ of $G_\nu$. Hence
\begin{equation}\label{eq.27}
{\mathcal O}^\tau(G_\nu)
\stackrel{\mathrm{def}}{=}
\{\operatorname{\widehat{Ad}}_g^*\tau, g\in G_\nu\}
={\Pi_{2}^{\mathfrak g}}({\mathcal O}^\sigma(G_\nu))
=\{(\operatorname{Ad}_g^*\sigma)|{\mathfrak g}_\nu,g\in G_\nu\}
\end{equation}
and
\begin{equation}\label{eq.28}
G_{\nu\tau}
=\{g\in G:\operatorname{Ad}^*_g\sigma|{\mathfrak n}
=\sigma|{\mathfrak n}=\nu,\
\operatorname{Ad}^*_g\sigma|{\mathfrak g}_\nu
=\sigma|{\mathfrak g}_\nu=\tau\}.
\end{equation}
Since by definition,
$\operatorname{Ad}^*(G_\nu)({\mathfrak n}+{\mathfrak g}_\nu)^\bot
=({\mathfrak n}+{\mathfrak g}_\nu)^\bot$, we have
\begin{equation}\label{eq.29}
G_{\nu\tau}=\{g\in G: \operatorname{Ad}^*_g({\mathcal A}_{\nu\tau})
={\mathcal A} _{\nu\tau}\}.
\end{equation}
Therefore by~\eqref{eq.24} the group
$G_{\nu\tau}$ contains the identity component
$N^0_\nu$ of $N_\nu$ and, moreover, the subgroup
$N_\nu^\mathrm{fin}\subset N_\nu$. The Lie algebra
${\mathfrak g}_{\nu\tau}$ of
$G_{\nu\tau}$ contains the Lie algebra
${\mathfrak n}_\nu$. Remark also that by definition
$G_\sigma\subset G_{\nu\tau}$ and
${\mathfrak g}_\sigma\subset{\mathfrak g}_{\nu\tau}$. Since
$N^0_\nu\subset G^0_{\nu\tau}\subset G_{\nu\tau}$,
the groups $\operatorname{Ad}^*(G^0_{\nu\tau})$ and
$\operatorname{Ad}^*(G_{\nu\tau})$
act transitively on the affine space
${\mathcal A}_{\nu\tau}$, that is
\begin{equation}\label{eq.30}
G^0_{\nu\tau}/(G_\sigma\cap G^0_{\nu\tau})\simeq
G_{\nu\tau}/G_\sigma\simeq N^0_\nu/N^0_\sigma\simeq
{\mathcal A}_{\nu\tau}
\end{equation}
and, consequently,
\begin{equation}\label{eq.31}
G_{\nu\tau}=N^0_\nu\cdot G_\sigma=G_\sigma\cdot N^0_\nu
\quad\text{and}\quad
{\mathfrak g}_{\nu\tau}={\mathfrak n}_\nu+{\mathfrak g}_\sigma.
\end{equation}
In particular,
\begin{equation}\label{eq.32}
\dim {\mathfrak g}_{\nu\tau}-\dim{\mathfrak g}_\sigma
=\dim{\mathfrak n}_\nu-\dim{\mathfrak n}_\sigma.
\end{equation}
Moreover, applying Lemma~\ref{le.1} to the spaces
in~\eqref{eq.30} we obtain that
\begin{equation}\label{eq.33}
G_\sigma\cap G^0_{\nu\tau}=G^0_\sigma,
\quad\pi_1(G^0_{\nu\tau})=\pi_1(G^0_\sigma)
\quad\text{and}\quad
G_{\nu\tau}/G^0_{\nu\tau}\simeq G_\sigma/G^0_\sigma.
\end{equation}
Also
$G^0_{\nu\tau}=N^0_\nu\cdot G^0_\sigma=G^0_\sigma\cdot N^0_\nu$.
But the group $N_\nu$ is a normal subgroup in
$G_\nu$ and, consequently, the group
$N_\nu^0$ is a normal subgroup in
$G^0_{\nu\tau}\subset G_\nu$. Similarly, by the definition the
group $N^0_\sigma=N^0_{\nu\nu}$ is a normal subgroup of
$G_\sigma$. Since this group is also a normal subgroup in
$N^0_\nu$, by~\eqref{eq.31} the group
$N^0_{\nu\nu}$ is a normal subgroup in
$G_{\nu\tau}$.

The group $N^0_\nu\subset G_{\nu\tau'}$ is the same group for all
$\tau'\in{\mathfrak g}_\nu$. The sum
${\mathcal A}_\nu=\bigcup_{\tau'\in{\mathcal
A}_\nu|{\mathfrak g}_\nu} {\mathcal A}_{\nu\tau'}$
is the union of the orbits of the group
$N^0_\nu$, the parallel affine subspaces of
${\mathcal A}_\nu$ with the associated vector space
$({\mathfrak n}+{\mathfrak g}_\nu)^\bot$.

Remark that
${\mathcal O}^\nu(G)$ is a disjoint union of coadjoint
orbits (isomorphic to
${\mathcal O}^\nu(N)\simeq N/N_\nu$) in the dual space
${\mathfrak n}^*$ and the group $G$
acts transitively on the set of these orbits. Moreover, by
equation~\eqref{eq.23} the dimension of
${\mathcal A}_{\nu\tau}$ is equal to the codimension of the
coadjoint orbit
${\mathcal O}^\nu(N)\subset {\mathfrak n}^*$ in the
$G$-orbit ${\mathcal O}^\nu(G)\subset{\mathfrak n}^*$.
The affine space ${\mathcal A}_{\nu\tau}$ as the orbit
${\mathcal O}^\sigma(N^0_\nu)\subset {\mathcal O}^\sigma(G)$
is an isotropic submanifold of the coadjoint orbit
${\mathcal O}^\sigma(G)$ (see relations~\eqref{eq.11} and
Remark~\ref{re.5}). We have proved
\begin{proposition}\label{pr.6}
The affine space ${\mathcal A}_{\nu\tau}$~\eqref{eq.22}
is an isotropic submanifold of the coadjoint orbit
${\mathcal O}^\sigma(G)\subset{\mathfrak g}^*$
containing the point $\sigma$ and
$\dim{\mathcal A}_{\nu\tau}=\dim{\mathcal
O}^\nu(G)-\dim{\mathcal O}^\nu(N)$.
The Lie subgroups $\operatorname{Ad}^*(N_\nu^0)$,
$\operatorname{Ad}^*(N_\nu^\mathrm{fin})$,
$\operatorname{Ad}^*(G_{\nu\tau}^0)$ and
$\operatorname{Ad}^*(G_{\nu\tau})$ of
$\operatorname{Ad}^*(G)$ preserve the affine subspace
${\mathcal A}_{\nu\tau}\subset{\mathfrak g}^*$.
The actions of these groups on
${\mathcal A}_{\nu\tau}$ are transitive. Moreover, the
orbits of the action of
$\operatorname{Ad}^*(N_\nu^0)$ on the affine subspace
${\mathcal A}_{\nu}\subset{\mathfrak g}^*$
are the parallel affine subspaces with the associated vector space
$({\mathfrak n}+{\mathfrak g}_\nu)^\bot$. The group
$N^0_{\nu\nu}$ is a normal subgroup of the Lie groups
$G_{\nu\tau}$, $N^0_\nu$ and topologically
$N^0_\nu/N^0_{\nu\nu}\simeq ({\mathfrak n}+{\mathfrak g}_\nu)^\bot$.
\end{proposition}

\begin{definition}\label{de.7}
The affine subspace
${\mathcal A}_{\nu\tau}=\sigma+({\mathfrak n}+{\mathfrak g}_\nu)^\bot$
contained in the coadjoint orbit
${\mathcal O}^\sigma(G)\subset{\mathfrak g}^*$ and denoted by
${\mathcal A}(\sigma,{\mathfrak n})$,
will be called the isotropic affine subspace associated with the ideal
${\mathfrak n}$ of ${\mathfrak g}$.
\end{definition}

\begin{remark}\label{re.8}
By relations~~\eqref{eq.23}, \eqref{eq.25} and~\eqref{eq.31} for any
$\sigma\in{\mathfrak g}^*$ the following conditions are
equivalent:
1)~${\mathcal A}(\sigma,{\mathfrak n})=\{\sigma\}$;
2)~$\dim{\mathcal A}(\sigma,{\mathfrak n})=0$;
3)~${\mathfrak g}_\nu+{\mathfrak n}={\mathfrak g}$;
4)~${\mathfrak g}_\sigma={\mathfrak g}_{\nu\tau}$;
5)~${\mathfrak n}_\nu\subset{\mathfrak g}_\sigma$.
Here, recall, $\nu=\sigma|{\mathfrak n}$ and
$\tau=\sigma|{\mathfrak g}_\nu$.
\end{remark}

\begin{remark}\label{re.9}
If $N$ is an affine algebraic Lie group, then its adjoint
representation $N\to GL({\mathfrak n})$,
$n\mapsto\operatorname{Ad}_n|{\mathfrak n}$, is a
${\mathbb F}$-morphism. In this case the affine algebraic
group $N_\nu$ always has a finite number of connected
(irreducible) components, and consequently,
$N_\nu^\mathrm{fin}=N_\nu\subset G_{\nu\tau}$.
Then by~\eqref{eq.31}
$G_{\nu\tau}=G_\sigma\cdot N_\nu$ and, consequently,
$G_{\nu\tau}/N_\nu\simeq G_\sigma/N_\sigma$.
We obtain the exact sequence
$$
e\to N_\sigma\to G_\sigma \to G_{\nu\tau}/N_\nu \to e,
$$
which generalizes Rawnsley's exact
sequence~\cite[Eq.(1)]{Ra} in the case of semidirect products.
\end{remark}

\begin{remark}\label{re.10}
The dual space
${\mathfrak n}^*$ is a Poisson manifold with the natural
linear Poisson structure and with the coadjoint orbits as
the corresponding symplectic leaves. Then the
$G$-orbit ${\mathcal O}^\nu(G)$
as the union of such (isomorphic) leaves is a Poisson submanifold of
${\mathfrak n}^*$. The Poisson structure on
${\mathcal O}^\nu(G)$ has constant rank
$\dim{\mathcal O}^\nu(N)$ and by Proposition~\ref{pr.6}
its corank equals
$\dim{\mathcal A}_{\nu\tau}=\dim({\mathfrak
g}_\nu+{\mathfrak n})^\bot$.
\end{remark}

\subsection{Reduced-group orbits and index of a Lie algebra}

We continue with the notation of the previous subsections.
Here we consider the orbit
${\mathcal O}^\tau(G_\nu)\subset{\mathfrak g}^*_\nu$
in more details. We will show that this orbit is the union
of disjoint coadjoint orbits of some reduced Lie algebra.

Indeed, as we remarked above (see~\eqref{eq.27}), the set
${\mathcal O}^\tau(G_\nu)$ consists of the restrictions
$\operatorname{Ad}^*_g\sigma|{\mathfrak g}_\nu$, where
$g\in G_\nu$. But by definition of the Lie group
$G_\nu$ we have
$\operatorname{Ad}^*_g\sigma|{\mathfrak n}=\nu$ for any
$g\in G_\nu$, that is, all elements of the orbit vanish on
the ideal ${\mathfrak n}^\natural_\nu$ of the Lie algebra
${\mathfrak g}_\nu$ (see~\eqref{eq.19}). Consider the
quotient algebra
${\mathfrak b}_\nu={\mathfrak g}_\nu/{\mathfrak n}^\natural_\nu$.
Since the connected subgroup of
$G_\nu$ corresponding to the subalgebra
${\mathfrak n}^\natural_\nu$ is not necessarily closed in
$G_\nu$, we will describe the coadjoint orbits of
${\mathfrak b}_\nu$ in terms of the Lie group
$G_\nu$.

Let ${\pi_\nu}:{\mathfrak g}_\nu\to{\mathfrak b}_\nu$
be the canonical homomorphism. The dual map
${\pi_\nu}^*:{\mathfrak b}_\nu^*\to{\mathfrak g}_\nu^*$
is a linear embedding and identifies the dual space
${\mathfrak b}_\nu^*$ naturally with the annihilator
$({\mathfrak n}^\natural_\nu)^{\bot_\nu}\subset {\mathfrak g}_\nu^*$
of ${\mathfrak n}^\natural_\nu$ in
${\mathfrak g}_\nu^*$. By Lemma~\ref{le.2} and by
relation~\eqref{eq.27} the set
\begin{equation}\label{eq.34}
{\mathcal O}^\tau=\{(\operatorname{Ad}_g^*\sigma)|{\mathfrak g}_\nu,
g\in G^0_\nu\}
\subset {\mathfrak b}_\nu^*\subset{\mathfrak g}^*_\nu
\end{equation}
is a coadjoint orbit in
${\mathfrak b}_\nu^*=({\mathfrak n}^\natural_\nu)^{\bot_\nu}$
passing through the element
$\tau\in{\mathfrak b}_\nu^*\subset{\mathfrak g}^*_\nu$.

In particular, ${\mathcal O}^\tau(G_\nu)$ is the union of disjoint
coadjoint orbits of the reduced Lie algebra
${\mathfrak b}_\nu$. This orbit
${\mathcal O}^\tau(G_\nu)$ will be called a reduced-group
orbit. Remark here that this group and this orbit are the
analog of Rawnsley's the little-group and the little-group
orbit in the case of semidirect products (see~\cite{Ra})).
Our term "reduced" is motivated by the reduction by stages
procedure of Marsden-Misio\l ek-Ortega-Perlmutter-Ratiu
(see~\cite{MMPR} and \cite{MMOPR}), where the
one-dimensional central extension of the quotient group
$G^0_\nu/N^0_\nu$ (with the Lie algebra
${\mathfrak b}_\nu$) is a natural symmetry group for the
second step of the reduction procedure (see also
Remark~\ref{re.20}).

By~\eqref{eq.32} we can replace
$\dim{\mathfrak n}_\nu-\dim{\mathfrak n}_\sigma$
in the left hand side of identity~\eqref{eq.17} by
$\dim {\mathfrak g}_{\nu\tau}-\dim{\mathfrak g}_\sigma$.
Therefore after simple rearrangements we obtain the identity
\begin{equation}\label{eq.35}
\begin{split}
\dim{\mathfrak g}_\sigma
&=[\dim{\mathfrak n}-\dim({\mathfrak g}/{\mathfrak g}_\nu)]
+\dim({\mathfrak g}_{\nu\tau}/{\mathfrak n}_\nu)\\
&=[\dim{\mathfrak n}-\dim({\mathfrak g}/{\mathfrak g}_\nu)]
+[\dim({\mathfrak g}_{\nu\tau}/{\mathfrak n}^\natural_\nu)
-\dim({\mathfrak n}_\nu/{\mathfrak n}^\natural_\nu)],
\end{split}
\end{equation}
where we recall that
$\sigma\in{\mathfrak g}^*$ is an arbitrary element and
$\nu=\sigma|{\mathfrak n}$ and
$\tau=\sigma|{\mathfrak g}_\nu$.

Let $\mu\in{\mathfrak n}^*$. Because
$\mu|{\mathfrak n}_\mu=0$ (i.e. $\mu\in{\mathfrak n}_\mu^\bot$)
if and only if
$\mu\in\operatorname{ad}^*_{\mathfrak n}\mu$,
and the function
$\mu\mapsto\dim({\mathbb F}\mu+\operatorname{ad}^*_{\mathfrak n}\mu)$
is lower semi-continuous on $R({\mathfrak n}^*)$, the set
$R^\natural({\mathfrak n}^*)
=\{\mu\in R({\mathfrak n}^*): \dim({\mathfrak n}_\mu
/{\mathfrak n}^\natural_\mu)=1\}$ is a
Zariski open subset of ${\mathfrak n}^*$.
Put $\delta^\natural({\mathfrak n})=1$ if this set is not empty,
and $\delta^\natural({\mathfrak n})=0$ otherwise.

\begin{remark}\label{re.}
If $R^\natural({\mathfrak n}^*)=\emptyset$ then
$\mu\in\operatorname{ad}^*_{\mathfrak n}\mu$ for all
$\mu\in R({\mathfrak n}^*)$ and, consequently,
the each coadjoint orbit of ${\mathfrak n}$ consisting
of ${\mathfrak n}$-regular elements with arbitrary
its element $\mu$ contains the set $\{z\mu\}$,
where $z\not=0$ if ${\mathbb F}={\mathbb C}$ and
$z>0$ if ${\mathbb F}={\mathbb R}$.
This follows from the fact that the coadjoint orbits in
$R({\mathfrak n}^*)\subset{\mathfrak n}^*$
are defined uniquely by the integrable vector subbundle
$\mu\mapsto\operatorname{ad}^*_{\mathfrak n}\mu$,
$\mu\in R({\mathfrak n}^*)$ (of constant corank
$\operatorname{ind}{\mathfrak n}$)
of the tangent bundle $TR({\mathfrak n}^*)$. It is clear that
$R^\natural({\mathfrak n}^*)\not=\emptyset$
($\delta^\natural({\mathfrak n})=1$) if the
algebra ${\mathfrak n}$ is semisimple
and $R^\natural({\mathfrak n}^*)=\emptyset$
($\delta^\natural({\mathfrak n})=0$) if
${\mathfrak n}$ is a Frobenius Lie algebra,
i.e. $\operatorname{ind}{\mathfrak n}=0$.
\end{remark}

Suppose that $\nu|{\mathfrak n}_\nu\not=0$, i.e
$\dim({\mathfrak n}_\nu/{\mathfrak n}^\natural_\nu)=1$.
In this case the extension
${\mathfrak b}_\nu={\mathfrak g}_\nu/{\mathfrak n}^\natural_\nu$
\begin{equation}\label{eq.36}
0\to{\mathfrak n}_\nu/{\mathfrak n}^\natural_\nu
\to{\mathfrak g}_\nu/{\mathfrak n}^\natural_\nu
\to{\mathfrak g}_\nu/{\mathfrak n}_\nu\to0
\end{equation}
is a one-dimensional central extension of the quotient algebra
${\mathfrak g}_\nu/{\mathfrak n}_\nu$. Let
${\mathcal B}_\nu$ be the image of the set
${\mathcal A}_\nu\subset{\mathfrak g}^*$
under the restriction map
\begin{equation}\label{eq.37}
{\Pi_{2}^{\mathfrak g}}:{\mathfrak g}^*\to{\mathfrak g}^*_\nu,
\quad \beta\mapsto\beta|{\mathfrak g}_\nu.
\end{equation}
Put ${\mathcal B}_0={\Pi_{2}^{\mathfrak g}}({\mathfrak n}^\bot)$.
It is easy to see that
${\mathcal B}_\nu\subset({\mathfrak n}^\natural_\nu)^{\bot_\nu}$,
where
$({\mathfrak n}^\natural_\nu)^{\bot_\nu}={\mathfrak b}^*_\nu$,
and ${\mathcal B}_0=({\mathfrak n}_\nu)^{\bot_\nu}$.
By dimension arguments,
$\dim{\mathcal B}_\nu=\dim({\mathfrak g}_\nu/{\mathfrak n}_\nu)$,
i.e. ${\mathcal B}_\nu$ is an affine subspace of codimension one
in ${\mathfrak b}^*_\nu$ and therefore
\begin{equation}\label{eq.38}
{\mathcal B}_\nu=\{\tilde\tau\in({\mathfrak n}^\natural_\nu)^{\bot_\nu}:
\tilde\tau|{\mathfrak n}_\nu=\nu|{\mathfrak n}_\nu\}
=\tau+({\mathfrak n}_\nu)^{\bot_\nu}.
\end{equation}
\begin{remark}\label{re.12}
The restriction
${\Pi_{2}^{\mathfrak g}}|{\mathcal A}_\nu$ of the linear map
${\Pi_{2}^{\mathfrak g}}$ is a bundle with the total space
${\mathcal A}_\nu$, the affine space
${\mathcal B}_\nu$ as its base and the space
$({\mathfrak g}_\nu+{\mathfrak n})^\bot\subset{\mathfrak g}^*$
as its fibre.
\end{remark}
By~\eqref{eq.8} and by
$G_\nu$-equivariance of the map
${\Pi_{2}^{\mathfrak g}}$ the space
${\mathcal B}_\nu$ is $G_\nu$-invariant
and is the union of coadjoint orbits of the Lie algebra
${\mathfrak b}_\nu$ and ${\mathfrak g}_\nu$
simultaneously (see expressions~\eqref{eq.27}
and~\eqref{eq.34}). Moreover, the union of disjoint
$G_\nu$-invariant affine subspaces
$\lambda{\mathcal B}_\nu={\mathcal B}_{\lambda\nu}$,
$\lambda\not=0$ of
${\mathfrak b}_\nu^*$ is an open dense subset in
${\mathfrak b}_\nu^*$:
\begin{equation}\label{eq.39}
{\Pi_{2}^{\mathfrak g}}\bigl(\bigcup_{\lambda\in
{\mathbb F}\setminus\{0\}}
\lambda{\mathcal A}_\nu\bigr)
=\bigcup_{\lambda\in{\mathbb F}\setminus\{0\}}\lambda{\mathcal B}_\nu
={\mathfrak b}_\nu^*\setminus{\mathcal B}_0,
\quad \dim{\mathfrak b}_\nu^*-\dim({\mathfrak n}_\nu)^{\bot_\nu}=1.
\end{equation}
Thus ${\mathcal B}_\nu$ and each of these affine spaces
$\lambda{\mathcal B}_\nu$ contain coadjoint orbits of the
Lie algebra ${\mathfrak b}_\nu$ of maximal dimension (as usual for
one-dimensional central extensions). Now as an immediate
consequence of identity~\eqref{eq.35} we obtain

\begin{theorem}\label{th.13}
Let ${\mathfrak g}$ be a Lie algebra over the field
${\mathbb F}$ and ${\mathfrak n}$ be its non-zero ideal.
Let $\nu$ be an element of $R^\natural({\mathfrak n}^*)$
if $R^\natural({\mathfrak n}^*)\not=\emptyset$ or
an element of $R({\mathfrak n}^*)$
if $R^\natural({\mathfrak n}^*)=\emptyset$ and such that
${\mathcal A}_\nu\cap R({\mathfrak g}^*)\not=\emptyset$.
Then $\operatorname{ind}{\mathfrak g}
=\operatorname{ind}({\mathfrak g},{\mathfrak n})
+(\operatorname{ind}{\mathfrak b}_\nu-\delta^\natural({\mathfrak n}))$,
where ${\mathfrak b}_\nu={\mathfrak g}_\nu/{\mathfrak n}^\natural_\nu$
and the ideal
${\mathfrak n}^\natural_\nu=\ker(\nu|{\mathfrak n}_\nu)$.
For any $\sigma\in{\mathfrak g}^*$ and
$\tau\in{\mathfrak g}^*_\nu$ such that
$\sigma|{\mathfrak n}=\nu$ and
$\sigma|{\mathfrak g}_\nu=\tau$ the element $\sigma $ is
${\mathfrak g}$-regular if and only if the element $\tau $ is
${\mathfrak b}_\nu$-regular.
\end{theorem}

\begin{remark}\label{re.14}
If $\nu|{\mathfrak n}_\nu\not=0$ then the one-dimensional
algebra ${\mathfrak n}_\nu/{\mathfrak n}^\natural_\nu$
is a subalgebra of the center of
${\mathfrak g}_\nu/{\mathfrak n}^\natural_\nu$.
Fix some element $z\in{\mathfrak n}_\nu$ such that
$\nu(z)=1$ and a splitting
$\ker\tau={\mathfrak s}_\nu\dotplus{\mathfrak n}^\natural_\nu$
of the kernel of $\tau\in{\mathfrak g}_\nu^*$. It is clear that
${\mathfrak g}_\nu={\mathfrak s}_\nu\dotplus{\mathbb F}
z\dotplus{\mathfrak n}^\natural_\nu$
and for arbitrary $\xi,\eta\in{\mathfrak g}_\nu$ the
${\mathbb F} z$-component of the commutator
$[\xi,\eta]$ is the vector $\tau([\xi,\eta])z$.
In other words, the central extension~\eqref{eq.36} is
determined by the map
${\mathfrak g}_\nu\times{\mathfrak g}_\nu\to{\mathbb F}$,
$(\xi,\eta)\mapsto\langle\tau,[\xi,\eta]\rangle$ on
${\mathfrak g}_\nu$ which factorizes to the cocycle
$\gamma_\tau$ on
${\mathfrak g}_\nu/{\mathfrak n}_\nu$ (by~\eqref{eq.19}
$[{\mathfrak g}_\nu,{\mathfrak n}_\nu]
\subset{\mathfrak n}^\natural_\nu$
and by definition $\tau|{\mathfrak n}_\nu=\nu|{\mathfrak n}_\nu$).
If $\nu|{\mathfrak n}_\nu=0$ then the map
$(\xi,\eta)\mapsto\langle\tau,[\xi,\eta]\rangle$ on
${\mathfrak g}_\nu$ factorizes to the trivial cocycle
$\gamma_\tau$ on the quotient algebra
${\mathfrak g}_\nu/{\mathfrak n}_\nu$. Note that the cocycle
$\gamma_\tau$ coincides with the cocycle constructed by
Panasyuk in~\cite{Pa} using direct calculations
and is independent of the choice of the extension $\tau$
of the covector $\nu|{\mathfrak n}_\nu$~\cite[Lemma 2.1]{Pa}.
\end{remark}

The dual space of a Lie algebra is a Poisson manifold with
the natural linear Poisson structure induced by the commutator
$[\cdot,\cdot]$. The coadjoint orbits in this space are the
corresponding symplectic leaves of the Poisson structure.
Since the affine subspace
${\mathcal B}_\nu$ is the union of the $G^0_\nu$-orbits in
${\mathfrak b}^*_\nu$ (the symplectic leaves),
${\mathcal B}_\nu$ is a Poisson submanifold of
${\mathfrak b}^*_\nu$. Fixing the origin
$\tau\in{\mathcal B}_\nu$ to identify
${\mathcal B}_\nu$ with the dual space
$({\mathfrak g}_\nu/{\mathfrak n}_\nu)^*=({\mathfrak n}_\nu)^{\bot_\nu}$
(see~\eqref{eq.38}), we fix some ``affine'' Poisson structure
$\eta_\tau$ on $({\mathfrak g}_\nu/{\mathfrak n}_\nu)^*$.
Remark that if the central extension~\eqref{eq.36} of the Lie algebra
${\mathfrak g}_\nu/{\mathfrak n}_\nu$
is trivial, this Poisson structure is equivalent to the
natural linear Poisson structure
$\eta_\mathrm{can}$ on the dual space to the Lie algebra
${\mathfrak g}_\nu/{\mathfrak n}_\nu$
(there will be a natural origin in
${\mathcal B}_\nu$). If $\nu|{\mathfrak n}_\nu=0$
then the (trivial) cocycle $\gamma_\tau$ determines
the trivial one-dimensional central extension of
${\mathfrak g}_\nu/{\mathfrak n}_\nu$ and, consequently,
the new Poisson structure $\eta_\tau\simeq\eta_\mathrm{can}$ on
$({\mathfrak g}_\nu/{\mathfrak n}_\nu)^*$.

Determine the index
$\operatorname{ind}({\mathfrak g}_\nu/{\mathfrak n}_\nu,\eta_\tau)$
of the Poisson structure
$\eta_\tau$ on $({\mathfrak g}_\nu/{\mathfrak n}_\nu)^*$
as the codimension of the symplectic leaf of the maximal
dimension in
$({\mathfrak g}_\nu/{\mathfrak n}_\nu)^*$.
It is clear that
$\operatorname{ind}({\mathfrak g}_\nu/{\mathfrak n}_\nu,\eta_\tau)
=\operatorname{ind}{\mathfrak b}_\nu-1$ if $\nu|{\mathfrak n}_\nu\not=0$
because $\operatorname{codim}{\mathcal B}_\nu$ in
${\mathfrak b}_\nu$ equals
$1$ and the symplectic leaves are coadjoint orbits of
${\mathfrak b}_\nu$.
Also $\operatorname{ind}({\mathfrak g}_\nu/{\mathfrak n}_\nu,\eta_\tau)
=\operatorname{ind}{\mathfrak b}_\nu$ if $\nu|{\mathfrak n}_\nu=0$
because in this case
${\mathfrak b}_\nu={\mathfrak g}_\nu/{\mathfrak n}_\nu$ and
$\eta_\tau\simeq\eta_\mathrm{can}$. We obtain the following
assertion of Panasyuk~\cite[Th. 2.7]{Pa}:
\begin{corollary}[Panasyuk's formula]\label{co.15}
Let ${\mathfrak g}$ be a Lie algebra and
${\mathfrak n}$ be its ideal. Then
$$
\operatorname{ind}{\mathfrak g}=\operatorname{ind}({\mathfrak g},
{\mathfrak n})+\operatorname{ind}({\mathfrak g}_\nu/{\mathfrak n}_\nu,
\eta_\tau),
$$
where $\nu\in{\mathfrak n}^*$ is a generic element.
\end{corollary}

Remark that Theorem~\ref{th.13} defines more precisely the
notion of the set of ``generic elements''. This set is
defined in~\cite{Pa} indirectly as an open dense subset in
${\mathfrak n}^*$ on which the function
$\nu\mapsto \operatorname{ind}({\mathfrak g}_\nu/{\mathfrak
n}_\nu,\eta_\gamma)$
is constant.

Assume that the ideal
${\mathfrak n}$ is Abelian and there exists a complementary
to ${\mathfrak n}$ subalgebra
${\mathfrak k}\subset{\mathfrak g}$, i.e.
${\mathfrak g}={\mathfrak k}\dotplus{\mathfrak n}$.
Then the Lie algebra ${\mathfrak g}$ is a semi-direct product of
${\mathfrak k}$ and the Abelian ideal
${\mathfrak n}$. It is evident that
${\mathfrak n}_\nu={\mathfrak n}$ for any non-zero
$\nu\in{\mathfrak n}^*$ and since
${\mathfrak n}_\nu\subset{\mathfrak g}_\nu$,
the isotropy subalgebra
${\mathfrak g}_\nu={\mathfrak k}_\nu\dotplus{\mathfrak n}$,
where ${\mathfrak k}_\nu={\mathfrak k}\cap{\mathfrak g}_\nu$.
But by~\eqref{eq.19}
$[{\mathfrak k}_\nu,{\mathfrak
n}^\natural_\nu]\subset{\mathfrak n}^\natural_\nu$,
where ${\mathfrak n}^\natural_\nu=\ker\nu$
and, consequently, the algebra
${\mathfrak b}_\nu={\mathfrak g}_\nu/{\mathfrak
n}^\natural_\nu =({\mathfrak k}_\nu+{\mathfrak
n}^\natural_\nu)/{\mathfrak n}^\natural_\nu\dotplus
{\mathfrak n}_\nu/{\mathfrak n}^\natural_\nu$
is a trivial one-dimensional central extension of
${\mathfrak k}_\nu\simeq {\mathfrak g}_\nu/{\mathfrak n}_\nu$.
Therefore
$\operatorname{ind}{\mathfrak
k}_\nu=\operatorname{ind}{\mathfrak b}_\nu-1$.
Note that
$\operatorname{ind}({\mathfrak g},{\mathfrak
n})=\operatorname{ind}({\mathfrak k},{\mathfrak n})$
because ${\mathfrak n}$ is Abelian. Since an element
$\nu\in{\mathfrak n}^*$ is
${\mathfrak g}$-regular if and only if this element is
${\mathfrak k}$-regular, we obtain the following assertion
of Ra\"{\i}s~\cite{Rais}:
\begin{corollary}[Ra\"{\i}s' formula]
Let the Lie algebra ${\mathfrak g}$ be a semi-direct product of
${\mathfrak k}$ and the Abelian ideal
${\mathfrak n}$. Let $\nu\in{\mathfrak n}^*$ be a
${\mathfrak k}$-regular element for which there exists a
${\mathfrak g}$-regular element
$\sigma\in{\mathfrak g}^*$ such that
$\sigma|{\mathfrak n}=\nu$. Then
$\operatorname{ind}{\mathfrak
g}=\operatorname{ind}({\mathfrak k},{\mathfrak
n})+\operatorname{ind}{\mathfrak k}_\nu$.
\end{corollary}

\subsection{The bundle of reduced-group orbits}

We retain to the general case when
${\mathfrak n}$ is an arbitrary ideal of ${\mathfrak g}$ and
$\sigma$ is an arbitrary element of
${\mathfrak g}^*$. Any element
${\hat\sigma}\in{\mathfrak g}^*$ determines a pair
$({\hat\nu},{\hat\tau})$, where
${\hat\nu}={\hat\sigma}|{\mathfrak n}$ and
${\hat\tau}={\hat\sigma}|{\mathfrak g}_{\hat\nu}$.
Such a pair is denoted by
${\Pi_{12}^{\mathfrak g}}({\hat\sigma})$. By the definition,
${\Pi_{12}^{\mathfrak
g}}({\hat\sigma}_1)={\Pi_{12}^{\mathfrak g}}({\hat\sigma}_2)$
if and only if
${\hat\sigma}_1,{\hat\sigma}_2\in {\mathcal A}_{{\hat\nu}{\hat\tau}}$
for some ${\hat\nu}\in{\mathfrak n}^*$ and
${\hat\tau}\in{\mathfrak g}^*_{\hat\nu}$.
In this case the elements
${\hat\sigma}_1,{\hat\sigma}_2$ belong to the same
$\operatorname{Ad}^*(G)$-orbit ${\mathcal O}$ in
${\mathfrak g}^*$ because the set
${\mathcal A}_{{\hat\nu}{\hat\tau}}$
is an orbit of the Lie subgroup
$N^0_{\hat\nu}\subset G$. Therefore the
$\operatorname{Ad}^*$-action of $G$ on the coadjoint orbit
${\mathcal O}$ induces the action of $G$ on the set
${\Pi_{12}^{\mathfrak g}}({\mathcal O})$.
We will show that on the set
${\Pi_{12}^{\mathfrak g}}({\mathcal O})$
there exists a structure of a smooth manifold such that the map
${\Pi_{12}^{\mathfrak g}}|{\mathcal O}$ is a
$G$-equivariant submersion. Remark also that for arbitrary
${\hat\tau}_0\in{\mathfrak g}^*_{\hat\nu}$ there exists some
${\hat\sigma}_0\in{\mathfrak g}^*$ such that
${\Pi_{12}^{\mathfrak g}}({\hat\sigma}_0)=({\hat\nu},{\hat\tau}_0)$
iff
${\hat\tau}_0|{\mathfrak n}_{{\hat\nu}}={\hat\nu}|{\mathfrak
n}_{{\hat\nu}}$.
In this case such an element
${\hat\tau}_0\in{\mathfrak g}^*_{\hat\nu}$ is called a
${\mathfrak g}^*_{\hat\nu}$-extension of
${\hat\nu}\in{\mathfrak n}^*$.

Let $B$ be the $G$-orbit in
${\mathfrak n}^*$ with respect to the action
$\rho^*$. Now we construct a bundle of redused-group orbits
over the orbit $B$. This bundle is the bundle
$p: P\to B$ such that the fibre
$F_P({{\hat\nu}})=p^{-1}({\hat\nu})$ is an orbit of
$G_{{\hat\nu}}$ in
${\mathfrak g}_{{\hat\nu}}^*$ passing through some
${\mathfrak g}^*_{\hat\nu}$-extension of
${\hat\nu}\in{\mathfrak n}^*$ and if
$g$ belongs to $G$ and ${\hat\tau}$ to
$F_P({{\hat\nu}})$ then
$g.{\hat\tau}\in F_P({g\cdot{\hat\nu}})$ (a right action)
is defined by $\langle g.{\hat\tau},\xi'\rangle=\langle
{\hat\tau},\operatorname{Ad}_g\xi' \rangle$,
where $\xi'\in{\mathfrak g}_{g\cdot{\hat\nu}}$. It follows that
$G$ acts transitively on $P$. We prove below that this
bundle and this action are smooth.

The bundle of reduced-group orbits may be described in
another way. Consider the (smooth) bundle
$P_{\nu\tau}=G\times_{G_\nu}(G_\nu/G_{\nu\tau})$,
the bundle associated to the principal bundle with base
$G/G_\nu$, total space $G$ and fibre
$(G_\nu/G_{\nu\tau})$. Here
$\nu$ denotes some element of the orbit $B$ and
$\tau\in F_P(\nu)$. Then
$F_P(\nu)\simeq G_\nu/G_{\nu\tau}$. The elements of
$P_{\nu\tau}$ are orbits of $G_\nu$ (on
$G\times(G_\nu/G_{\nu\tau})$), where the action on the right
is given by $(g,[h]).h'=(gh',[{h'}^{-1}h])$ with $g\in G$,
$h,h'\in G_\nu$
($[h]=h\,G_{\nu\tau}\in G_\nu/G_{\nu\tau}$). The element
$(g,[h]).G_\nu$ of $P_{\nu\tau}$, is identified with the point
$(gh)^{-1}.\tau$ in $F_P({(gh)^{-1}\cdot\nu})$. Defining $p$ by
$p((g,[h]).G_\nu)=(gh)^{-1}\cdot\nu$ and the right action of
$G$ on $P$ by $g'.(g,[h]).G_\nu= (g'^{-1}g,[h]).G_\nu$ makes
$p:P\to B$ a smooth bundle of reduced-group orbits over
$B={\mathcal O}^\nu(G)$. The following proposition
generalizes Proposition 1 from~\cite{Ra}.
\begin{proposition}\label{pr.17}
There is a bijection between the set of bundles of
reduced-group orbits and the set of coadjoint orbits of
$G$ on ${\mathfrak g}^*$.
\end{proposition}
\begin{proof}
Let $p:P\to B$ be a bundle of reduced-group orbits, take
$\nu\in B$, $\tau\in F_P(\nu)$ and choose some extension
$\sigma\in{\mathfrak g}^*$ with
$\sigma|{\mathfrak n}=\nu\in{\mathfrak n}^*$ and
$\sigma|{\mathfrak g}_\nu=\tau\in{\mathfrak g}_\nu^*$. If
${\mathcal O}^\sigma$ is the
$\operatorname{Ad}^*(G)$-orbit through $\sigma$ in
${\mathfrak g}^*$ then it depends only on
$p:P\to B$ but not of the choices made because all extensions
of $(\nu,\tau)$ are elements of this orbit (see~\eqref{eq.30}).

Conversely, let ${\mathcal O}$ be an
$\operatorname{Ad}^*(G)$-orbit in
${\mathfrak g}^*$ and $\sigma$ a point of
${\mathcal O}$ with $\sigma|{\mathfrak n}=\nu$ and
$\sigma|{\mathfrak g}_\nu=\tau$.
Construct the bundle of reduced-group over
$B$, the orbit of $\nu$ in
${\mathfrak n}^*$, with fibre $F_P(\nu)$, the
$G_\nu$-orbit of $\tau\in{\mathfrak g}_\nu^*$ in
${\mathfrak g}_\nu^*$. This gives a bundle depending only on
${\mathcal O}$ and not of the choices made. These two
constructions are the inverses of each other and set up the
required bijection.
\end{proof}

If we have an orbit
${\mathcal O}^\sigma={\mathcal O}^\sigma(G)$ in
${\mathfrak g}^*$ and the associated bundle
$p:P\to B$, then the following diagram (on the left) of
$G$-equivariant maps is commutative. Recall that, by
the definition,
${\Pi_{12}^{\mathfrak
g}}({\hat\sigma})=({\hat\sigma}|{\mathfrak
n},{\hat\sigma}|{\mathfrak g}_\nu)=({\hat\nu},{\hat\tau})$
and
${\Pi_1^{\mathfrak g}}({\hat\sigma})={\hat\sigma}|{\mathfrak n}$.
\begin{equation}\label{eq.40}
\begin{array}{rrl}
{\mathcal O}^\sigma&&\\
{\makebox[10pt]{$\scriptstyle{\Pi_{12}^{\mathfrak g}}$}\downarrow}
&\makebox[0pt]{\makebox[30pt][r]{{$\searrow$}
{\makebox[12pt]{$\scriptstyle{\Pi_1^{\mathfrak g}}$}}}}& \\
P&\stackrel{p}\longrightarrow&B
\end{array}
\hskip50pt
\begin{array}{crl}
G\times_{G_\nu}(G_\nu/G_\sigma)
&&\\
{\makebox[10pt]{$\scriptstyle{\Pi_{12}^{\mathfrak g}}$}\downarrow}
&\makebox[0pt]{\makebox[30pt][r]{{$\searrow$}%
{\makebox[12pt]{$\scriptstyle{\Pi_1^{\mathfrak g}}$}}}}& \\
G\times_{G_\nu}(G_\nu/G_{\nu\tau})
&\stackrel{p}\longrightarrow&G/G_\nu
\end{array}
\end{equation}
As we remarked above, the fibres of
${\Pi_{12}^{\mathfrak g}}$ are affine subspaces of
${\mathfrak g}^*$ whose associated vector space is
$({\mathfrak n}+{\mathfrak g}_{\hat\nu})^\bot$
(in general there will be no natural origin in
${\Pi_{12}^{\mathfrak
g}}^{-1}({\hat\nu},{\hat\tau})={\mathcal A}_{{\hat\nu}{\hat\tau}}$).
Thus the fibres of ${\Pi_{12}^{\mathfrak g}}$ are the orbits on
${\mathcal O}^\sigma$ of the groups conjugated to
$G_{{\hat\nu}{\hat\tau}}$.

The map ${\Pi_{12}^{\mathfrak g}}:{\mathcal O}^\sigma(G)\to P$
is smooth because the map
${\Pi_1^{\mathfrak g}}:{\mathcal O}^\sigma(G)\to B$
is a submersion and the left diagram is commutative. This
fact can be established also by identifying
$G$-equivariantly the bundle $P$ with
$P_{\nu\tau}=G\times_{G_\nu}(G_\nu/G_{\nu\tau})$.
But by the definition, ${\mathcal O}^\sigma\simeq G/G_\sigma$ and
$B\simeq G/G_\nu$. Consider the space
$G\times_{G_\nu}(G_\nu/G_\sigma)$,
where the action on the right is given by
$(g,hG_\sigma).h'=(gh',{h'}^{-1}hG_\sigma)$ with
$g$ in $G$, $h, h'$ in $G_\nu$. The standard map
$$
G\times_{G_\nu}(G_\nu/G_\sigma)\to G/G_\sigma,\quad
[(g,hG_\sigma)]_{G_\nu}\mapsto ghG_\sigma
$$
is a $G$-equivariant diffeomorphism with respect to the natural
left actions of
$G$. Therefore, using this identification, we obtain the
following expressions for the
$G$-equivariant maps
$p,{\Pi_1^{\mathfrak g}},{\Pi_{12}^{\mathfrak g}}$:\
$p([(g,hG_{\nu\tau})]_{G_\nu})=ghG_\nu$,
$$
{\Pi_1^{\mathfrak g}}([(g,hG_\sigma)]_{G_\nu})=ghG_\nu,
\quad\text{and}\quad
{\Pi_{12}^{\mathfrak g}}([(g,hG_\sigma)]_{G_\nu})
=[(g,hG_{\nu\tau})]_{G_\nu}.
$$
It is clear that the diagram above (on the right) is also
commutative and these two diagrams are equivalent. Remark
also that by Proposition~\ref{pr.6} the fibre
${\mathcal A}_{\nu\tau}$ is an isotropic submanifold of the
coadjoint orbit ${\mathcal O}^\sigma(G)$. We have proved
\begin{theorem}\label{th.18}
The map
${\Pi_{12}^{\mathfrak g}}:{\mathcal O}^\sigma(G)\to P$ is a
$G$-equivariant submersion of the coadjoint orbit
${\mathcal O}^\sigma$ onto the bundle $P$
of reduced-group orbits. This map is a bundle with the total space
${\mathcal O}^\sigma$, the base $P$ and the affine space
${\mathcal A}_{\nu\tau}$ (the isotropic submanifold of
${\mathcal O}^\sigma(G)$) as its fibre. The commutative
diagrams~\eqref{eq.40} are equivalent.
\end{theorem}

\subsection{Isotropic affine subspaces of coadjoint orbits}

As follows from Proposition~\ref{pr.6} each coadjoint orbit
${\mathcal O}$ of the Lie algebra ${\mathfrak g}$
contains the isotropic affine subspace associated with its ideal
${\mathfrak n}$. We will show below that if this affine
subspace is trivial then any isotropic affine subspace of
the corresponding coadjoint orbit in
${\mathfrak b}_\nu^*\subset{\mathfrak g}_\nu^*$
determines some isotropic affine subspace of
${\mathcal O}$.

Let ${\mathcal O}^\sigma={\mathcal O}^\sigma(G)$, where
$\sigma\in{\mathfrak g}^*$, be a coadjoint orbit in
${\mathfrak g}^*$. Consider also the coadjoint orbit
${\mathcal O}^\tau$ in ${\mathfrak g}^*_\nu$~\eqref{eq.34}
passing through the element
$\tau\in{\mathfrak g}_\nu^*$, where
$\tau=\sigma|{\mathfrak g}_\nu$.
To simplify the notation, this orbit
${\mathcal O}^\tau={\mathcal O}^\tau(G^0_\nu)$
of the connected Lie group
$G^0_\nu$ will be considered as an orbit of the (closed) Lie
subgroup $G_\nu^\bullet=G_\nu^0\cdot G_{\nu\tau}$ of
$G_\nu$, containing the whole isotropy subgroup
$G_{\nu\tau}$. The $\sigma$-orbit
${\mathcal O}^\sigma(G^\bullet_\nu)\subset {\mathcal O}^\sigma(G_\nu)$
in ${\mathfrak g}^*$ is also connected because,
by~\eqref{eq.31},
$G_\nu^\bullet=G_\nu^0\cdot G_\sigma$. Recall that
${\Pi_{2}^{\mathfrak g}}$ denotes the natural
$G_\nu$-equivariant projection
${\mathfrak g}^*\to{\mathfrak g}_\nu^*$,
$\beta\mapsto\beta|{\mathfrak g}_\nu$
defined by~\eqref{eq.37}.
\begin{proposition}\label{pr.19}
Let $\sigma\in{\mathfrak g}^*$ be an arbitrary element and
$\nu=\sigma|{\mathfrak n}$,
$\tau=\sigma|{\mathfrak g}_\nu$. The restriction
$p_2={\Pi_{2}^{\mathfrak g}}|{\mathcal O}^\sigma(G^\bullet_\nu)$
of the projection
${\Pi_{2}^{\mathfrak g}}$ is a
$G^\bullet_\nu$-equivariant submersion of the orbit
${\mathcal O}^\sigma(G^\bullet_\nu)$
onto the coadjoint orbit
${\mathcal O}^\tau$ in ${\mathfrak g}_\nu^*$. This map
$p_2:{\mathcal O}^\sigma(G^\bullet_\nu)\to {\mathcal O}^\tau$
is a bundle with the total space
${\mathcal O}^\sigma(G^\bullet_\nu)$, the coadjoint orbit
${\mathcal O}^\tau$ as its base and the affine space
${\mathcal A}_{\nu\tau}\simeq G_{\nu\tau}/G_\sigma$
as its fibre. Moreover,
$p_2^*(\omega')=\omega|{\mathcal O}^\sigma(G^\bullet_\nu)$,
where $\omega'$ and $\omega$ are the canonical
Kirillov-Kostant-Souriau symplectic 2-forms on the coadjoint orbits
${\mathcal O}^\tau\subset{\mathfrak g}^*_\nu$ and
${\mathcal O}^\sigma(G)\subset{\mathfrak g}^*$ respectively.
\end{proposition}
\begin{proof}
To prove the first part of the proposition it is sufficient
to remark that
$$
{\mathcal O}^\sigma(G^\bullet_\nu)
\simeq G^\bullet_\nu/G_\sigma, \quad
{\mathcal O}^\tau
={\mathcal O}^\tau(G^\bullet_\nu)
\simeq G^\bullet_\nu/G_{\nu\tau}, \quad
{\mathcal A}_{\nu\tau}
\simeq G_{\nu\tau}/G_\sigma
$$
and
$G^\bullet_\nu/G_\sigma=
G^\bullet_\nu\times_{G_{\nu\tau}}(G_{\nu\tau}/G_\sigma)$
with the standard right action of
$G_{\nu\tau}$.

By $G^\bullet_\nu$-equivariance of the map
$p_2$, we have
$p_{2*}(\sigma)(\operatorname{ad}^*_{\xi}\sigma)
=\widehat{\operatorname{ad}}^*_{\xi}\tau$
for $\xi,\eta\in{\mathfrak g}_\nu$.
Then, by the definition, of the form $\omega'$,
$$
(p_2^*\omega')(\sigma)(\operatorname{ad}^*_\xi\sigma,
\operatorname{ad}^*_\eta\sigma)
=\omega'(\tau)(\widehat{\operatorname{ad}}^*_\xi\tau,
\widehat{\operatorname{ad}}^*_\eta\tau)
=\tau([\xi,\eta])=\sigma([\xi,\eta]).
$$
Taking into account the expression~\eqref{eq.10} for
$\omega$ at the point $\sigma$ and
$G^\bullet_\nu$-invariance of the forms $\omega$ and
$\omega'$, we complete the proof.
\end{proof}

\begin{remark}\label{re.20}
The proposition above admits the following moment map
interpretation which is motivated by Panasyuk's
approach~\cite{Pa}. Indeed, the identity map
$J_G:{\mathcal O}^\sigma(G)\to{\mathfrak g}^*$,
${\hat\sigma}\mapsto{\hat\sigma}$
is an equivariant moment map for
$\operatorname{Ad}^*$-action of $G$ on
${\mathcal O}^\sigma(G)$. Since
${\mathfrak n}$ is a subalgebra of ${\mathfrak g}$, the map
$J_N={\Pi_1^{\mathfrak g}}\circ J_G$ of
${\mathcal O}^\sigma(G)$ into ${\mathfrak n}^*$,
${\hat\sigma}\mapsto{\hat\sigma}|{\mathfrak n}$
is an equivariant moment map for the restricted action of
$N\subset G$ on
${\mathcal O}^\sigma(G)$. Then by~\eqref{eq.6} the set
$J_N^{-1}(\nu)={\mathcal A}_\nu\cap{\mathcal
O}^\sigma(G)={\mathcal O}^\sigma(G_\nu)$
is a submanifold of ${\mathcal O}^\sigma(G)$. If
$N^\mathrm{fin}_\nu=N_\nu$, then by Proposition~\ref{pr.6}
the quotient space
$J^{-1}_N(\nu)/N_\nu\simeq{\Pi_{2}^{\mathfrak g}}({\mathcal
O}^\sigma(G_\nu))$
is a reduced symplectic manifold. This manifold is the orbit
${\mathcal O}^\tau(G_\nu)\subset{\mathfrak
b}^*_\nu\subset{\mathfrak g}^*_\nu$,
a union of disjoint coadjoint orbits (connected components)
in the reduced Lie algebra
${\mathfrak b}_\nu^*$. The reduced symplectic structure on
${\mathcal O}^\tau(G_\nu)$ coincides with the canonical
Kirillov-Kostant-Souriau symplectic form on each connected
component of ${\mathcal O}^\tau(G_\nu)$.
\end{remark}

\begin{proposition}\label{pr.21}
We retain the notation of Proposition~\ref{pr.19}. Suppose
that the coadjoint orbit
${\mathcal O}^\tau\subset{\mathfrak g}_\nu^*$
contains the isotropic affine subspace
${\mathcal I}(\tau)$ passing through the point $\tau$. If
$\dim {\mathcal A}_{\nu\tau}=0$, then the projection
$p_2:{\mathcal O}^\sigma(G^0_\nu)\to{\mathcal O}^\tau$
is a bijection and the preimage
${\mathcal I}(\sigma)=p_2^{-1}({\mathcal I}(\tau))$,
${\mathcal I}(\sigma)\subset {\mathcal
O}^\sigma(G^0_\nu)\subset{\mathcal A}_\nu$,
is an isotropic affine subspace of the coadjoint orbit
${\mathcal O}^\sigma(G)$ passing through
$\sigma\in{\mathfrak g}^*$.
\end{proposition}

\begin{proof}
It is an immediate consequence of Proposition~\ref{pr.19}
that the map $p_2$ is a bijection and
${\mathcal I}(\sigma)=p_2^{-1}({\mathcal I}(\tau)$
is an isotropic submanifold of the coadjoint orbit
${\mathcal O}^\sigma(G)$. Let us show that the set
${\mathcal I}(\sigma)$ is an affine subspace of
${\mathfrak g}^*$.

Since ${\mathcal I}(\tau)$ is an affine subspace of
${\mathfrak g}_\nu^*$, this space contains any line
$\tau+t\alpha$, where
$\alpha\in{\mathfrak g}_\nu^*$ is a non-zero tangent vector to
${\mathcal I}(\tau)$ at
$\tau$. There exists a unique tangent vector
$\alpha'\in T_\sigma{\mathcal O}^\sigma(G_\nu^0)$ such that
$p_{2*}(\sigma)(\alpha')=\alpha$ because
$p_2$ is a bijection. Moreover,
$\alpha'\in{\mathfrak n}^\bot$ because
${\mathcal O}^\sigma(G_\nu^0)\subset{\mathcal A}_\nu$ and
$\alpha={\Pi_{2}^{\mathfrak g}}(\alpha')$ because
$p_2$ is a restriction of the linear map
${\Pi_{2}^{\mathfrak g}}$~\eqref{eq.26}.

Since ${\mathcal I}(\tau)\subset{\mathcal O}^\tau$, for any
$t\in{\mathbb F}$ there exists $g\in G_\nu^0$ such that
$\tau+t\alpha=\widehat{\operatorname{Ad}}^*_g\tau$.
Now to complete the proof it is sufficient to show that the point
$\sigma'={\operatorname{Ad}}^*_{g^{-1}}\sigma+
t{\operatorname{Ad}}^*_{g^{-1}}\alpha'$
coincides with $\sigma$. Indeed, by $G_\nu$-equivariance
of the linear map ${\Pi_{2}^{\mathfrak g}}$
$$
{\Pi_{2}^{\mathfrak g}}(\sigma')=
{\Pi_{2}^{\mathfrak g}}({\operatorname{Ad}}^*_{g^{-1}}\sigma+
t{\operatorname{Ad}}^*_{g^{-1}}\alpha')
=\widehat{\operatorname{Ad}}^*_{g^{-1}}\tau+
t\widehat{\operatorname{Ad}}^*_{g^{-1}}\alpha
=\tau.
$$
But $\sigma'\in{\mathcal A}_\nu$ because
$\operatorname{Ad}^*(G_\nu)({\mathcal A}_\nu)={\mathcal A}_\nu$
and $\sigma+t\alpha'\in{\mathcal A}_\nu$. Therefore
$\sigma'$ belongs to the one-point set
${\mathcal A}_{\nu\tau}=\{\sigma\}$.
\end{proof}

\begin{proposition}\label{pr.22}
We retain the notation of Proposition~\ref{pr.19}. Suppose that
$\dim{\mathcal A}_{\nu\tau}=0$ and the quotient algebra
${\mathfrak b}_\nu={\mathfrak g}_\nu/{\mathfrak n}^\natural_\nu$,
${\mathfrak n}^\natural_\nu=\ker(\nu|{\mathfrak n}_\nu)$
is Abelian. Then
\begin{enumerate}
\item[$1)$]
${\mathcal O}^\sigma(G)={\mathcal O}^\sigma(N)$ and
${\mathcal O}^\nu(G)={\mathcal O}^\nu(N)$, where
${\mathcal O}^\sigma(G)$ and ${\mathcal O}^\nu(N)$
are the coadjoint orbits of the Lie algebras
${\mathfrak g}$ and ${\mathfrak n}$ respectively;
\item[$2)$] the projection
$p_1:{\mathcal O}^\sigma(G)\to{\mathcal O}^\nu(N)$,
$\sigma'\mapsto\sigma'|{\mathfrak n}$, is a symplectic
$G$-equivariant covering map with the discrete fiber
$\simeq N_\nu/N_\sigma$ and
${\mathfrak g}_\sigma={\mathfrak g}_\nu$,
${\mathfrak n}_\sigma={\mathfrak n}_\nu$;
\item[$3)$] if $N_\nu=N_\nu^\mathrm{fin}$, then
$p_1$ is a diffeomorphism, and, in particular,
$G_\nu=G_\sigma$, $N_\nu=N_\sigma$.
\end{enumerate}
\end{proposition}

\begin{proof}
Since $N$ is a normal subgroup of $G$, the $G$-orbit
${\mathcal O}^\nu(G)$ is a disjoint union of isomorphic
$N$-orbits. These $N$-orbits are open subsets because
$\dim{\mathcal O}^\nu(G)-\dim{\mathcal O}^\nu(N)
=\dim{\mathcal A}_{\nu\tau}=0$.
Then ${\mathcal O}^\nu(G)={\mathcal O}^\nu(N)$ because
$G$ is connected.

By Proposition~\ref{pr.19}
$\dim{\mathcal O}^\sigma(G_\nu)=\dim{\mathcal O}^\tau(G_\nu)$
because ${\mathcal A}_{\nu\tau}=\{\sigma\}$.
Since each connected component of
${\mathcal O}^\tau(G_\nu)$ is a coadjoint orbit of the Lie
algebra ${\mathfrak b}_\nu$ which is Abelian,
$\dim{\mathcal O}^\tau(G_\nu)=0$.
Thus by Lemma~\ref{le.4} the
$G$-equivariant map
$p_1:{\mathcal O}^\sigma(G)\to{\mathcal O}^\nu(G)$
is a bundle with the discrete fibre
${\mathcal O}^\sigma(G_\nu)$. Taking into account the
identity ${\mathcal O}^\nu(G)={\mathcal O}^\nu(N)$ and the
$N$-equivariance of the local diffeomorphism
$p_1$ we obtain that
$T_\sigma{\mathcal O}^\sigma(G)=T_\sigma{\mathcal O}^\sigma(N)$,
i.e. the orbit ${\mathcal O}^\sigma(N)$ is an open subset of
${\mathcal O}^\sigma(G)$. Using the same arguments as above,
we obtain that
${\mathcal O}^\sigma(G)={\mathcal O}^\sigma(N)$. Since
${\mathcal O}^\sigma(G)\simeq N/N_\sigma$ and
${\mathcal O}^\nu(G)\simeq N/N_\nu$, the fiber
${\mathcal O}^\sigma(G_\nu)\simeq N_\nu/N_\sigma$. Since
$\dim{\mathcal O}^\sigma(G)=\dim{\mathcal O}^\nu(G)$,
$\dim{\mathfrak g}_\sigma=\dim{\mathfrak g}_\nu$. Thus
${\mathfrak g}_\sigma={\mathfrak g}_\nu$ because
${\mathfrak g}_\sigma\subset{\mathfrak g}_\nu$.

The local diffeomorphism
$p_1$ is symplectic with respect to the canonical
symplectic structures on the both coadjoint orbits. To prove
this fact it is sufficient to observe that
$T_\sigma{\mathcal
O}^\sigma(G)=\operatorname{ad}^*_{\mathfrak n}\sigma$,
$p_{1*}(\sigma)(\operatorname{ad}^*_{\xi}\sigma)
=\widetilde{\operatorname{ad}}^*_{\xi}\nu$
and $\sigma([\xi,\eta])=\nu([\xi,\eta])$ for any
$\xi,\eta\in{\mathfrak n}$ (by $N$-equivariance of
$p_1$), and to use definition~\eqref{eq.10} of the canonical
symplectic form.

If $N_\nu=N_\nu^\mathrm{fin}$, then by Proposition~\ref{pr.6}
the group $\operatorname{Ad}^*(N_\nu)$ preserves the one-point set
${\mathcal A}_{\nu\tau}=\{\sigma\}$ and, consequently,
$N_\nu=N_\sigma$. Hence
$p_1$ is a diffeomorphism.
\end{proof}

\subsection{Coadjoint orbits in general position}

In the previous subsection we considered arbitrary
coadjoint orbits of
${\mathfrak g}$. Now we consider the structure of the orbits
in general position. To this end put
\begin{equation}\label{eq.41}
\operatorname{co}({\mathfrak g},{\mathfrak n})
=\operatorname{ind}{\mathfrak n}
-\operatorname{ind}({\mathfrak g},{\mathfrak n}).
\end{equation}
Taking into account that
$\dim{\mathcal O}^\nu(G)-\dim{\mathcal O}^\nu(N)
=\operatorname{codim}_{{\mathfrak n}^*}{\mathcal O}^\nu(N)
-\operatorname{codim}_{{\mathfrak n}^*}{\mathcal O}^\nu(G)$
for any $\nu\in{\mathfrak n}^*$, we can interpret the number
$\operatorname{co}({\mathfrak g},{\mathfrak n})$
as a ``complexity'' of the action of
$N\subset G$ on homogeneous spaces of
$G$ in general position. Then by~\eqref{eq.23}
\begin{equation}\label{eq.42}
\operatorname{co}({\mathfrak g},{\mathfrak n})
=\dim({\mathfrak n}+{\mathfrak g}_{\nu})^\bot
=\dim{\mathcal A}(\sigma,{\mathfrak n})
\end{equation}
for all $\nu$ from some dense subset of ${\mathfrak n}^*$
containing the non-empty Zariski open set of all
${\mathfrak g}$-regular and
${\mathfrak n}$-regular points of ${\mathfrak n}^*$. Here
$\sigma\in{\mathcal A}_\nu=({\Pi_1^{\mathfrak g}})^{-1}(\nu)$
and
${\mathcal A}(\sigma,{\mathfrak n})
=\sigma+\dim({\mathfrak n}+{\mathfrak g}_\nu)^\bot$
is the isotropic affine subspace of the coadjoint orbit
${\mathcal O}^\sigma\subset{\mathfrak g}^*$.

The case when $\operatorname{co}({\mathfrak g},{\mathfrak n})=0$
we consider in more detail.
\begin{lemma}\label{le.23}
Suppose that
$\operatorname{co}({\mathfrak g},{\mathfrak n})=0$. Let
$\nu\in{\mathfrak n}^*$ be any
${\mathfrak n}$-regular point. Then
\begin{enumerate}
\item[$1)$] $\nu\in{\mathfrak n}^*$ is a
${\mathfrak g}$-regular point; \item[$2)$]
${\mathcal O}^{\nu}(G)={\mathcal O}^{\nu}(N)$
or, equivalently,
${\mathfrak g}_{\nu}+{\mathfrak n}={\mathfrak g}$; \item
[$3)$] if
${\mathcal A}_{\nu}\cap R({\mathfrak g}^*)\not=\emptyset$
and the quotient algebra
${\mathfrak b}_{\nu}={\mathfrak g}_{\nu}/{\mathfrak n}^\natural_{\nu}$
is Abelian then for each
${\mathfrak n}$-regular point
$\nu_1\in{\mathfrak n}^*$ {\rm (i)} the algebra
${\mathfrak b}_{\nu_1} ={\mathfrak g}_{\nu_1}/{\mathfrak
n}^\natural_{\nu_1}$
is Abelian; {\rm (ii)}
${\mathcal A}_{\nu_1}\subset R({\mathfrak g}^*)$;
{\rm (iii)}
${\mathfrak g}_{\sigma_1}={\mathfrak g}_{\nu_1}$, where
$\sigma_1\in {\mathcal A}_{\nu_1}$;
{\rm (iv)} the Lie algebra
${\mathfrak g}_{\nu_1}$ is Abelian; {\rm (v)} there exists
an Abelian Lie algebra
${\mathfrak a}\subset{\mathfrak g}_{\nu_1}$
such that the Lie algebra
${\mathfrak g}$ is a semidirect product of
${\mathfrak a}$ and the ideal ${\mathfrak n}$, i.e.
${\mathfrak g}={\mathfrak a}\ltimes{\mathfrak n}$.
\end{enumerate}
\end{lemma}
\begin{proof}
Since $\operatorname{co}({\mathfrak g},{\mathfrak n})=0$,
$\dim{\mathcal O}^{\nu_0}(G)=\dim{\mathcal O}^{\nu_0}(N)$
for some point $\nu_0\in R({\mathfrak n}^*)$ which is
${\mathfrak g}$-regular. Hence
${\mathfrak g}_{\nu_0}+{\mathfrak n}={\mathfrak g}$.
But for each $\nu_1\in R({\mathfrak n}^*)$ the isotropy algebra
${\mathfrak n}_{\nu_1}={\mathfrak g}_{\nu_1}\cap{\mathfrak n}$
has constant dimension
$\operatorname{ind}{\mathfrak n}$ and
$\dim{\mathfrak g}_{\nu_1}\geqslant\dim{\mathfrak g}_{\nu_0}$.
Therefore
${\mathfrak g}_{\nu_1}+{\mathfrak n}={\mathfrak g}$, i.e.
$\dim{\mathfrak g}_{\nu_1}=\dim{\mathfrak g}_{\nu_0}$.
In particular, ${\mathfrak g}_{\nu}+{\mathfrak n}={\mathfrak
g}$.

If the quotient algebra
${\mathfrak b}_{\nu}={\mathfrak g}_{\nu}/{\mathfrak n}^\natural_{\nu}$
is Abelian then the coadjoint orbit
${\mathcal O}^{\tau}$ is a one-point set $\{\tau\}$, i.e.
${\mathfrak g}_{\nu\tau}={\mathfrak g}_{\nu}$. Since
${\mathcal A}_{\nu}\cap R({\mathfrak g}^*)\not=\emptyset$,
there exists a ${\mathfrak g}$-regular element
$\sigma\in{\mathfrak g}^*$ such that its restriction
$\sigma|{\mathfrak n}=\nu$. But by Remark~\ref{re.8},
${\mathfrak g}_{\nu\tau}={\mathfrak g}_{\sigma}$ and
${\mathfrak g}_{\nu_1\tau_1}={\mathfrak g}_{\sigma_1}$,
where $\sigma_1\in{\mathcal A}_{\nu_1}$ and
$\tau_1=\sigma_1|{\mathfrak g}_{\nu_1}$. Hence
$$
\dim{\mathfrak g}_{\sigma_1}
=\dim{\mathfrak g}_{\nu_1\tau_1}
\leqslant\dim{\mathfrak g}_{\nu_1}
=\dim{\mathfrak g}_{\nu}
=\dim{\mathfrak g}_{\nu\tau}
=\dim{\mathfrak g}_{\sigma}
=\operatorname{ind}{\mathfrak g}
$$
because ${\mathfrak g}_{\nu_1\tau_1}\subset{\mathfrak g}_{\nu_1}$.
But by definition
$\dim{\mathfrak
g}_{\sigma_1}\geqslant\operatorname{ind}{\mathfrak g}$.
Therefore ${\mathfrak g}_{\sigma_1}={\mathfrak g}_{\nu_1}$ and
${\mathcal A}_{\nu_1}\subset R({\mathfrak g}^*)$
(all these points are
${\mathfrak g}$-regular). The Lie algebra
${\mathfrak g}_{\sigma_1}$ is Abelian as an isotropy algebra
of a ${\mathfrak g}$-regular element of the coadjoint
representation (one can prove this fact differentiating the identity
$\langle\sigma_t,[{\mathfrak g}_{\sigma_t},{\mathfrak
g}_{\sigma_t}]\rangle=0$
using definition~\eqref{eq.2} of
${\mathfrak g}_{\sigma_t}$). Hence the algebra
${\mathfrak g}_{\nu_1}={\mathfrak g}_{\sigma_1}$
is Abelian. Since
${\mathfrak g}_{\nu_1}+{\mathfrak n}={\mathfrak g}$,
there exists a subspace
${\mathfrak a}\subset{\mathfrak g}_{\nu_1}$ for which
${\mathfrak g}={\mathfrak a}\dotplus{\mathfrak n}$.
This subspace is an Abelian subalgebra of
${\mathfrak g}$.
\end{proof}

\subsection{Integral orbits: a necessary but
non sufficient condition}

In this subsection we will use the notation of the previous
subsections, but suppose in addition that the ground field
${\mathbb F}$ is the field
${\mathbb R}$ of real numbers.

First of all we will give an exposition of some results of
Kostant~\cite[\S\S 5.6, 5.7, Theorem 5.7.1]{Ko} on the
geometry of coadjoint orbits.

Let $H$ be a connected Lie group with the Lie algebra
${\mathfrak h}$. Fix some covector
$\varphi\in{\mathfrak h}^*$ and consider the coadjoint orbit
${\mathcal O}^\varphi={\mathcal O}^\varphi(H)\simeq H/H_\varphi$
in ${\mathfrak h}^*$. We will say that the coadjoint orbit
${\mathcal O}^\varphi$ in the dual space
${\mathfrak h}^*$ is integral if its canonical symplectic
form is integral, i.e. this form determines an integral
cohomology class in
$H^2({\mathcal O}^\varphi,{\mathbb Z})\subset H^2({\mathcal
O}^\varphi,{\mathbb R})$.

Denote by $H_\varphi^\sharp$ the set (possibly empty) of all
characters $\chi: H_\varphi\to{\mathbb S}^1\subset{\mathbb C}$
such that
$d\chi(e)=2\pi i\cdot\varphi|{\mathfrak h}_\varphi$, where
${\mathfrak h}_\varphi$ is the Lie algebra of the isotropy
group $H_\varphi$. For such a character
$\chi\in H_\varphi^\sharp$,
\begin{equation}\label{eq.43}
\chi(\exp\xi)=\exp(2\pi\, i\cdot\langle\varphi ,\xi\rangle)
\quad\text{for all}\quad \xi\in{\mathfrak h}_\varphi.
\end{equation}
Since the identity component
$H^0_\varphi$ of $H_\varphi$ is generated by its
neighborhood of the unity, the restriction
$\chi|H^0_\varphi$ is defined uniquely by
equation~\eqref{eq.43}. Therefore if
$H_\varphi^\sharp$ is not empty $H_\varphi^\sharp$ is a
$\pi^*_{H_\varphi/H^0_\varphi}$-principal
homogeneous space, where
$\pi^*_{H_\varphi/H^0_\varphi}$ is the group of
${\mathbb S}^1$-valued characters of the quotient group
$H_\varphi/H^0_\varphi$. In this case
$|H_\varphi^\sharp|=|\pi^*_{H_\varphi/H^0_\varphi}|$~\cite{Ko}.

Let $\widetilde H$ be the connected simply connected Lie group
with the Lie algebra
${\mathfrak h}$, the universal covering group of the
connected Lie group $H$ and $\tilde p:\widetilde H\to H$
be the corresponding covering homomorphism. Then
${\mathcal O}^\varphi=\widetilde H/\widetilde H_\varphi$,
where $\widetilde H_\varphi$ is the isotropy group of the element
$\varphi\in{\mathfrak h}^*$. By definition
$\widetilde H_\varphi=\tilde p^{-1}(H_\varphi)$ and
$H_\varphi\simeq\widetilde H_\varphi/D$, where
$D$ is the kernel of the restricted homomorphism
$\tilde p|\widetilde H_\varphi$.
The following Kostant's theorem~{\cite[Theorem 5.7.1]{Ko}}
is crucial for the forthcoming considerations.
\begin{theorem}[B.Kostant]\label{th.24}
The orbit ${\mathcal O}^\varphi$ in ${\mathfrak h}^*$ is integral
if and only if the character set $\widetilde H_\varphi^\sharp$
is not empty.
\end{theorem}

Remark that one can not formulate the integrality condition
for the orbit
${\mathcal O}^\varphi$ only in terms of the connected Lie
group $H_\varphi$ (defining this orbit) because as it will be
shown below (see Example~\ref{ex.25}) in the general case
the characters $\chi\in \widetilde H_\varphi^\sharp$
are not constant on the closed discrete subgroup
$D$ of the center of
$\widetilde H_\varphi$. In other words, it is possible that
$H_\varphi^\sharp=\emptyset$ while
$\widetilde H_\varphi^\sharp\not=\emptyset$.

\begin{example}\label{ex.25}
Consider the connected Lie group
$H=SO(3)$ and its universal covering group
$\widetilde H=SU(2)$ with the Lie algebra
${\mathfrak h}=su(2)$. Using the invariant scalar product
$\langle\varphi_1,\varphi_2\rangle=-\frac12\operatorname{Tr}
\varphi_1\varphi_2$
on ${\mathfrak h}$ we can identify the spaces
${\mathfrak h}$ and
${\mathfrak h}^*$. It is evident that for
$\varphi=\operatorname{diag}(ib,-ib)\in su(2)$ with
$b\in{\mathbb R}$ the isotropy group
$\widetilde H_\varphi=\{\operatorname{diag}(e^{ia},e^{-ia}),\
a\in{\mathbb R}\}$ and the isotropy algebra
${\mathfrak h}_\varphi=\{\operatorname{diag}(ia,-ia),\
a\in{\mathbb R}\}$.
In particular, $\widetilde H_\varphi$ contains the element
$-E=\operatorname{diag}(-1,-1)\in SU(2)$
of the kernel of the covering homomorphism
$\tilde p: SU(2)\to SO(3)$. Under our identification of
${\mathfrak h}$ with ${\mathfrak h}^*$ the map~\eqref{eq.43}
$\tilde\chi:\exp({\mathfrak h}_\varphi)\to{\mathbb S}^1$,
$\operatorname{diag}(e^{ia},e^{-ia})\mapsto e^{2\pi i ab}$,
is well defined if and only if
$2\pi b\in{\mathbb Z}$. Since the group
$\widetilde H_\varphi$ is connected, by Theorem~\ref{th.24}, the
orbit ${\mathcal O}^\varphi$ is integral if and only if the number
$2\pi b$ is integer. For such a covector
$\varphi$ the set $\widetilde H_\varphi^\sharp$
contains a unique element, the character
$\tilde\chi$. But if the number $2\pi b$ is odd then
$\tilde\chi(-E)=-1$. For such a covector $\varphi$ the set
$H_\varphi^\sharp$ is empty while
$\tilde H_\varphi^\sharp\not=\emptyset$.
Indeed, in the opposite case for
$\chi\in H_\varphi^\sharp$ we have by definition that
$\chi\circ\tilde p\in \tilde H_\varphi^\sharp$. Therefore
$\chi\circ\tilde p=\tilde\chi$. But
$(\chi\circ\tilde p)(-E)=1$ while
$\tilde\chi(-E)=-1$, the contradiction.
\end{example}

The character $\chi|H^0_\varphi$, $\chi\in H_\varphi^\sharp$ on
$H^0_\varphi$ admits another interpretation in terms of
differential forms. Choose a contractible neighborhood
$U\subset H^0_\varphi$ of the unity for which all
intersections $U\cap hU$, $h\in H^0_\varphi$ are also (smoothly)
contractible (one uses, for instance, a convex set relative
to any invariant Riemannian structure on
$H^0_\varphi$). The left $H^0_\varphi$-invariant one-form
$\theta_\varphi$ with
$\theta_\varphi(e)=\varphi|{\mathfrak h}_\varphi$
on the Lie group $H^0_\varphi$ is closed because, by the
definition~\eqref{eq.2} of an isotropy algebra,
$\varphi([{\mathfrak h}_\varphi,{\mathfrak h}_\varphi])=0$.
Therefore a character on
$H^0_\varphi$ determined by~\eqref{eq.43} exists if and only
if the one-form $\theta_\varphi$ is integral, i.e.
$\theta_\varphi\in H^1(H^0_\varphi,{\mathbb Z})$.
In this case there exists a family of local functions
$\{f_h: hU\to{\mathbb R}, h\in H^0_\varphi\}$ such that
$df_h=\theta_\varphi$ on the open subset $hU$ and
$f_{h_1}-f_{h_2}\in{\mathbb Z}$ if
$h_1U\cap h_2U\not=\emptyset$, $h_1,h_2\in H^0_\varphi$. By
$H^0_\varphi$-invariance of the form
$\theta_\varphi$ the family
$\{f_h\}$ determines the character on
$H^0_\varphi$ if $f_e(e)=0$. Then
$\chi|hU=\exp(2\pi if_h)$ and
\begin{equation}\label{eq.44}
f_h-\bigl(l_{h^{-1}}^*f_e+(1/2\pi i)\ln\chi(h)\bigr)\in{\mathbb Z}
\quad\text{on } hU,
\text{ where } l_{h^{-1}}(h')=h^{-1}h'.
\end{equation}
In this case we will say that the character
$\chi|H^0_\varphi$ is associated with the (integer) form
$\theta_\varphi$.

\begin{proposition}\label{pr.26}
Let $\sigma$ be an arbitrary element of
${\mathfrak g}^*$,
$\nu=\sigma|{\mathfrak n}$ and
$\tau=\sigma|{\mathfrak g}_\nu$.
There is a bijection between the sets
$G_{\nu\tau}^\sharp$ and $G_\sigma^\sharp$, where
$G_{\nu\tau}^\sharp$ denotes the set of all characters
$\chi: G_{\nu\tau}\to{\mathbb S}^1\subset{\mathbb C}$
such that $d\chi(e)=2\pi i\cdot\tau|{\mathfrak g}_\nu$.
This bijection is induced by the restriction map
$\chi\mapsto \chi|G_\sigma$.
\end{proposition}
\begin{proof}
Note that $\tau=\sigma|{\mathfrak g}_\nu$ and
$G_\sigma\subset G_{\nu\tau}$. But
${\mathfrak g}_\sigma\subset{\mathfrak g}_{\nu\tau}$, thus
$\tau|{\mathfrak g}_{\sigma}=\sigma|{\mathfrak g}_\sigma$
and by the definition for any
$\chi\in G_{\nu\tau}^\sharp$ we have
$\chi|G_\sigma\in G_\sigma^\sharp$.
Therefore, in order to prove the proposition it is
sufficient to show that each character
$\psi\in G_\sigma^\sharp$ admits an extension to some
character $\chi\in G_{\nu\tau}^\sharp$. This extension is unique
because by~\eqref{eq.33} the groups
$G_{\nu\tau}/G^0_{\nu\tau}$ and
$G_\sigma/G^0_\sigma$ are isomorphic and, in particular,
$\pi^*_{G_{\nu\tau}/G^0_{\nu\tau}}\simeq \pi^*_{G_\sigma/G^0_\sigma}$.

Consider now a character
$\psi\in G_\sigma^\sharp$. Since
$G^0_\sigma$ is a closed subgroup of
$G^0_{\nu\tau}$, we can choose a contractible neighborhood
$U\subset G^0_{\nu\tau}$ of the unity such that all
intersections $U\cap hU$,
$h\in G^0_\sigma$, are also contractible and, in addition,
$U\cap hU\not=\emptyset$ if and only if
$U\cap hU\cap G^0_\sigma\not=\emptyset$ (there exists a local
cross section $S\subset G^0_{\nu\tau}$ such that the map
$(s,g)\mapsto sg$,
$S\times G^0_\sigma\to S G^0_\sigma\subset G^0_{\nu\tau}$
is a diffeomorphism). Let
$\theta_\tau$ be a
$G^0_{\nu\tau}$-invariant one-form on the Lie group
$G^0_{\nu\tau}$ such that
$\theta_\tau(e)=\tau|{\mathfrak g}_{\nu\tau}$.
Since the form $\theta_\tau$ is closed, there exists a function
$f_e: U\to{\mathbb R}$ such that
$df_e=\theta_\tau|U$, $f_e(e)=0$. Put
$f_h=l_{h^{-1}}^*f_e+(1/2\pi i)\ln\psi(h)$ for all
$h\in G^0_\sigma\setminus\{e\}$. Then
$df_h=\theta_\tau|hU$ because the one-form
$\theta_\tau$ is $G^0_\sigma$-invariant.
Thus the difference $f_{h_1}-f_{h_2}$ on the set
$h_1U\cap h_2U\not=\emptyset$ is a real constant.

On the other hand, the co-vector
$\sigma|{\mathfrak g}_\sigma$ determines the left
$G^0_\sigma$-invariant one-form $\theta_\sigma$ on the group
$G^0_\sigma$. By the definition,
$\theta_\sigma$ coincides with the restriction
$\theta_\tau|G^0_\sigma$ and
$\psi|G^0_\sigma$ is the character associated with this form
$\theta_\sigma\in H^1(G^0_\sigma,{\mathbb Z})$.
Therefore from~\eqref{eq.44} it follows that the difference
$f_{h_1}-f_{h_2}$ is an integer constant on some nonempty
subset $h_1U\cap h_2U\cap G^0_\sigma$ and, consequently, on the
whole open set $h_1U\cap h_2U$. In other words, the function
$\chi_e: G^0_\sigma\cdot U\to{\mathbb S}^1$ given by
$\chi_e|hU=\exp(2\pi i f_h)$ is a well defined extension of
the function $\psi|G^0_\sigma$ onto the open set
$G^0_\sigma\cdot U\supset G^0_\sigma$.

Put ${\mathbf U}=G^0_\sigma\cdot U$.
Considering the family of functions
$\{l_{g^{-1}}^*f_h\}$ (for which
$d(l_{g^{-1}}^*f_h)=\theta_\tau|ghU$), we obtain that
$l_{g^{-1}}^*\chi_e=s\chi_e$ on
$g{\mathbf U}\cap{\mathbf U}\not=\emptyset$, where
$s$ is some constant factor from ${\mathbb S}^1$.
But by~\eqref{eq.30} and~\eqref{eq.33} the space
$G^0_{\nu\tau}/G^0_\sigma\simeq{\mathcal A}_{\nu\tau}$
is contactable. Therefore there exists a character
$\chi^0$ on $G^0_{\nu\tau}$ which is an extension of
$\psi|G^0_\sigma$ and which is associated with the one-form
$\theta_\tau$. Moreover, $\chi^0(\tilde gg\tilde g^{-1})=\chi^0(g)$ for
any (fixed) $\tilde g\in G_{\nu\tau}$ and for all $g\in G^0_{\nu\tau}$.
Indeed, putting $F(g)=\chi^0(\tilde gg\tilde g^{-1})$ and
$a_{\tilde g}: g\mapsto \tilde gg\tilde g^{-1}$ on $G^0_{\nu\tau}$,
we obtain that
$$
\frac{1}{2\pi i}\cdot\frac{dF}{F}
=\frac{1}{2\pi i}\cdot a_{\tilde g}^*\biggl(\frac{d\chi^0}{\chi^0}\biggr)
=a_{\tilde g}^*\theta_\tau=\theta_\tau
=\frac{1}{2\pi i}\cdot \biggl(\frac{d\chi^0}{\chi^0}\biggr)
$$
because by~\eqref{eq.28}
$\langle\tau,\operatorname{Ad}_{\tilde g}\xi\rangle
=\langle\sigma,\operatorname{Ad}_{\tilde g}\xi\rangle=
\langle\tau,\xi\rangle$
for all $\xi\in{\mathfrak g}_{\nu\tau}\subset{\mathfrak g}_\nu$.
Since $F(e)=\chi^0(e)$, we have $F=\chi^0$.

Taking into account that
$G_{\nu\tau}=G_\sigma\cdot G^0_{\nu\tau}$,
$G^0_{\nu\tau}\cap G_\sigma=G^0_\sigma$ (see~\eqref{eq.31})
and $\psi=\chi^0|G^0_\sigma$
we obtain that the map $\chi:G_{\nu\tau}\to{\mathbb S}^1$,
$\chi(hg)=\psi(h)\chi^0(g)$, where
$h\in G_\sigma$ and $g\in G^0_{\nu\tau}$,
is well defined. This map determines a character on
$G_{\nu\tau}$ because
$\chi^0(hgh^{-1})=\chi^0(g)$
for all $h\in G_\sigma\subset G_{\nu\tau}$ and
$g\in G^0_{\nu\tau}$. Finally, $\chi$ belongs to the set
$G_{\nu\tau}^\sharp$ because $\chi|G^0_{\nu\tau}=\chi^0$.
\end{proof}

Remark that Proposition~\ref{pr.26} generalizes Rawnley's
Proposition~2 from~\cite{Ra}.

\begin{proposition}\label{pr.27}
Let $\sigma\in{\mathfrak g}^*$ and
$\nu=\sigma|{\mathfrak n}$. An integrality of the coadjoint
orbit ${\mathcal O}^\tau\subset {\mathfrak g}^*_\nu$
is a necessary condition for an integrality of the coadjoint orbit
${\mathcal O}^\sigma\subset{\mathfrak g}^*$.
In general, this condition is not sufficient for an integrality of
${\mathcal O}^\sigma$.
\end{proposition}
\begin{proof}
If the form $\omega$ on
${\mathcal O}^\sigma={\mathcal O}^\sigma(G)$
is integral, then its restriction
$\omega|{\mathcal O}^\sigma(G^\bullet_\nu)$
to the submanifold
${\mathcal O}^\sigma(G^\bullet_\nu)\subset{\mathcal O}^\sigma(G)$
is also integral. Since by Proposition~\ref{pr.19} the map
$p_2:{\mathcal O}^\sigma(G^\bullet_\nu)\to {\mathcal O}^\tau$
is a locally trivial fibering with a contractible fibre, the
affine space ${\mathcal A}_{\nu\tau}$, the map
$p_2^*:\Lambda^2({\mathcal O}^\tau)\to \Lambda^2({\mathcal
O}^\sigma(G^*_\nu))$
induces an isomorphism
$H^2({\mathcal O}^\tau,{\mathbb Z})\to H^2({\mathcal
O}^\sigma(G^\bullet_\nu),{\mathbb Z})$.
Since by Proposition~\ref{pr.19}
$p_2^*(\omega')=\omega|{\mathcal O}^\sigma(G^\bullet_\nu)$,
the canonical symplectic form
$\omega'$ on ${\mathcal O}^\tau$
is integral and we obtain the first assertion of the
proposition.

Remark also that the first assertion of the proposition
follows also from Proposition~\ref{pr.26}. Indeed, we can
assume without restricting the generality that
$G$ is a connected and simply connected Lie group with the
Lie algebra
${\mathfrak g}$. By Theorem~\ref{th.24}, the character set
$G_\sigma^\sharp$ is not empty. By Proposition~\ref{pr.26},
$G^\sharp_{\nu\tau}\not=\emptyset$. Let
$\widetilde{G^0_\nu}$ be the universal covering group of the
connected group $G^0_\nu$ (with the Lie algebra
${\mathfrak g}_\nu$). By Theorem~\ref{th.24} the coadjoint orbit
${\mathcal O}^\tau$ is integral if and only if
$(\widetilde{G^0_\nu})_\tau^\sharp\not=\emptyset$.
However, the covering homomorphism
$\widetilde{G^0_\nu}\to G^0_\nu$ induces the homomorphism
$(\widetilde{G^0_\nu})_\tau\to (G^0_\nu)_\tau$
and, consequently,
$(\widetilde{G^0_\nu})_\tau^\sharp\not=\emptyset$ if
$(G^0_\nu)_\tau^\sharp\not=\emptyset$. Therefore
$(\widetilde{G^0_\nu})_\tau^\sharp\not=\emptyset$, because
$G^\sharp_{\nu\tau}\not=\emptyset$ and
$(G^0_\nu)_\tau$ is an open subgroup of
$G_{\nu\tau}$.

The second assertion of the proposition will be proven in
the next subsection showing that the converse is not
necessarily true. More precisely, we will construct the Lie algebra
${\mathfrak g}$ which is a semi-direct product of some Lie
subalgebra
${\mathfrak k}\subset{\mathfrak g}$ and the Abelian ideal
${\mathfrak n}$ and choose two coadjoint orbits
${\mathcal O}^\tau\subset{\mathfrak g}^*_\nu$ and
${\mathcal O}^\sigma\subset{\mathfrak g}^*$
which are not integral simultaneously while
$\tau=\sigma|{\mathfrak g}_\nu$.
\end{proof}

\begin{remark}\label{re.28}
All connected components of the reduced-group orbit
${\mathcal O}^\tau(G_\nu)$ are coadjoint orbits of the Lie
algebras ${\mathfrak g}_\nu$ and
${\mathfrak b}_\nu$ (under the identification of
${\mathfrak b}_\nu^*$ with
$({\mathfrak n}^\natural_\nu)^{\bot_\nu}\subset{\mathfrak g}_\nu^*$,
see~\eqref{eq.34} and Lemma~\ref{le.2}). These orbits
are simultaneously either integral or non-integral.
\end{remark}

\subsection{Split extensions using Abelian algebras (semidirect products)}
\label{ss.2.7}

In this subsection we will finish the proof of
Proposition~\ref{pr.27}. To this end we construct a
connected and simply connected Lie group
$G$ and construct some coadjoint orbit
${\mathcal O}^\sigma(G)$ in
${\mathfrak g}^*$ such that the set
$(G^0_\nu)_\tau^\sharp$ is empty while the coadjoint orbit
${\mathcal O}^\tau={\mathcal O}^\tau(G_\nu^0)$
is integral. Then by Proposition~\ref{pr.26} the set
$(G_\sigma)^\sharp$ is also empty, i.e. the orbit
${\mathcal O}^\sigma(G)$ is not integral.

Let $K$ be a connected and simply connected Lie group with the
Lie algebra ${\mathfrak k}$, and for $k$ in $K$ and
$f$ in the dual ${\mathfrak k}^*$ of ${\mathfrak k}$, let
$\operatorname{Ad}_k^*f$ denote the coadjoint action of
$k$ on $f$. If $\delta$ is a representation of
$K$ on a real, finite-dimensional space $V$, let
$d\delta$ be the corresponding tangent representation of
${\mathfrak k}$.

We can form the semi-direct product
$G=K\ltimes_\delta V$ using the representation
$\delta$ and identifying
$V$ with its group of translations. Then the Lie group
$G$ can be taken as $K\times V$ with multiplication
$(k_1,v_1)(k_2,v_2)=(k_1k_2,v_1+k_1\cdot v_2)$ for
$k_j\in K$, $v_j\in V$ and the algebra
${\mathfrak g}={\mathfrak k}\ltimes_{d\delta}V$ of
$G$ can be taken as
${\mathfrak k}\dotplus V$ with the Lie bracket
$$
[(\zeta_1,y_1),(\zeta_2,y_2)]
=([\zeta_1,\zeta_2],\zeta_1\cdot y_2-\zeta_2\cdot y_1)
$$
for $\zeta_j$ in ${\mathfrak k}$ and $y_j$ in
$V$. Here $k_j\cdot v_j=\delta(k_j)(v_j)$ and
$\zeta_j\cdot y_j=d\delta(\zeta_j)(y_j)$. Since
$(k,v)^{-1}=(k^{-1},-k^{-1}\cdot v)$, the adjoint action of
$G$ on ${\mathfrak g}$ is given by
\begin{equation}\label{eq.45}
\operatorname{Ad}_{(k,v)}(\zeta,y)
=(\operatorname{Ad}_k\zeta,
k\cdot y-(\operatorname{Ad}_k\zeta)\cdot v).
\end{equation}
The dual ${\mathfrak g}^*$ of
${\mathfrak g}$ can be identified with
${\mathfrak k}^*\times V^*$ and the coadjoint action of
$G$ on ${\mathfrak g}^*$ is given by
\begin{equation}\label{eq.46}
\langle\operatorname{Ad}^*_{(k,v)}(f',\nu'),(\zeta,y)\rangle=
\langle\operatorname{Ad}^*_k f',\zeta\rangle
-\langle \nu',(\operatorname{Ad}_k\zeta)\cdot v\rangle
+\langle k^*\cdot \nu',y \rangle,
\end{equation}
where $f'$ is in ${\mathfrak k}^*$ and $\nu'$ in
$V^*$; also, by the definition,
$\langle k^*\cdot \nu',y \rangle=\langle \nu',k\cdot y\rangle$.
Note that all above formulas for semidirect products are
standard up to notation (see for example,~\cite[\S2]{Ra}
or~\cite[\S2]{Ba}).

The subgroup
$N=\{(e,v)\in G, v\in V\}$ is a normal commutative subgroup
of $G$ with the Lie algebra
${\mathfrak n}=\{(0,y), y\in V\}$. The
$\operatorname{Ad}$-action~\eqref{eq.45} of $G$ on
${\mathfrak n}$ induces the action of $G$ on
${\mathfrak n}^*$:
$(k,v)\cdot\nu=k^*\cdot \nu$. Therefore for
$\nu\in{\mathfrak n}^*=V^*$
$$
G_\nu=\{(k,v)\in K\ltimes_\delta V,\ k^*\cdot \nu=\nu\}
\quad\text{and}\quad
{\mathfrak g}_\nu=\{(\zeta,y)\in {\mathfrak k}\ltimes_{d\delta} V,\
\zeta^*\cdot\nu=0\},
$$
that is $G_\nu=K_\nu\ltimes_\delta V$ and
${\mathfrak g}_\nu={\mathfrak k}_\nu\ltimes_{d\delta} V$,
where $K_\nu$ is the isotropy group of
$\nu\in V^*$ with Lie algebra
${\mathfrak k}_\nu=\{\zeta\in{\mathfrak k}: \zeta^*\cdot\nu=0\}$.
It is easy to verify using~\eqref{eq.46} that
$G_\nu$ is the stabilizer of the affine subspace
${\mathcal A}_\nu=\{(f,\nu), f\in{\mathfrak k}^*\}$.

Putting $\sigma=(f,\nu)$ and
$\tau=\sigma|{\mathfrak g}_\nu$, we obtain that
$\tau=(\varphi,\nu)$, where
$\varphi=f|{\mathfrak k}_\nu$. By definition~\eqref{eq.28},
the Lie group
\begin{equation}\label{eq.47}
\begin{split}
G_{\nu\tau}
&=\{(k,v)\in K_\nu\ltimes_\delta V:
\bigl(\operatorname{Ad}^*_{(k,v)}(f,\nu)\bigr)|{\mathfrak g}_\nu
=(f,\nu)|{\mathfrak g}_\nu\}\\
&=\{(k,v)\in K_\nu\ltimes_\delta V:
(\operatorname{Ad}_k^*f)|{\mathfrak k}_\nu=f|{\mathfrak k}_\nu\},
\end{split}
\end{equation}
because $k^*\cdot\nu=\nu$, $\zeta^*\cdot\nu=0$
and $\operatorname{Ad}_k\zeta\in{\mathfrak k}_\nu$
for all $k\in K_\nu$, $\zeta\in{\mathfrak k}_\nu$.
In particular, $\zeta\cdot v\in\ker\nu$ if $v\in\ker\nu$.
In other words,
$G_{\nu\tau}=K_{\nu\varphi}\ltimes_\delta V$, where
\begin{equation}\label{eq.48}
K_{\nu\varphi}=\{k\in K_\nu: \langle\varphi,
\operatorname{Ad}_k\zeta\rangle
=\langle\varphi,\zeta\rangle,\ \forall\zeta\in{\mathfrak k}_\nu\}.
\end{equation}

Suppose now that the group
$K_\nu$ is connected. Let $\widetilde{K_\nu}$
be its universal covering group with the covering homomorphism
$\tilde p_\nu:\widetilde{K_\nu}\to K_\nu$. Then
$\widetilde{G_\nu}=\widetilde{K_\nu}\ltimes_{\tilde\delta}V$
is the universal covering group of
$G_\nu$, where the semi-direct product if determined by the
representation
$\tilde\delta=\delta\circ\tilde p_\nu$. Since the group
$G_\nu$ is connected, the coadjoint orbit
${\mathcal O}^\tau\subset{\mathfrak g}^*_\nu$ is the orbit
${\mathcal O}^\tau(G_\nu)\simeq G_\nu/G_{\nu\tau}$.
But this orbit is also an orbit of
$\widetilde{G_\nu}$, that is
${\mathcal O}^\tau\simeq\widetilde{G_\nu}/(\widetilde{G_\nu})_\tau$.
It is easy to verify using expressions similar
to~\eqref{eq.46} and~\eqref{eq.47} that
$(\widetilde{G_\nu})_\tau =({\widetilde{K_\nu})_\varphi}
\ltimes_{\tilde\delta}V$, where
$(\widetilde{K_\nu})_\varphi=\tilde p^{-1}_\nu(K_{\nu\varphi})$.

Now we will establish bijections between the sets
$(\widetilde{G_\nu})_\tau^\sharp$ and
$(\widetilde{K_\nu})_\varphi^\sharp$,
$G_{\nu\tau}^\sharp$ and $K_{\nu\varphi}^\sharp$
using Rawnsley's formula~\cite[Eq.(2)]{Ra}. Indeed, for any character
$\psi\in K_{\nu\varphi}^\sharp$ the function
$\chi(k,v)=\psi(k)\exp(2\pi i\langle \nu,v\rangle)$
on the group $G_{\nu\tau}=K_{\nu\varphi}\ltimes_\delta V$
is a character because
$k^*\cdot\nu=\nu$. By~\eqref{eq.43} this character
$\chi$ is a unique extension of $\psi$ such that
$\chi\in G_{\nu\tau}^\sharp$. Thus there is a bijection
between $G_{\nu\tau}^\sharp$ and $K_{\nu\varphi}^\sharp$.
Using similar arguments one establishes a bijection between
$(\widetilde{G_\nu})_\tau^\sharp$ and
$(\widetilde{K_\nu})_\varphi^\sharp$ because
$(\widetilde{G_\nu})_\tau
=({\widetilde{K_\nu})_\varphi}\ltimes_{\tilde\delta}V$.
By Proposition~\ref{pr.26} and Theorem~\ref{th.24}, the orbit
${\mathcal O}^\tau(G_\nu)$ is integral and the orbit
${\mathcal O}^\sigma(G)$ is not integral if and only if
$(\widetilde{G_\nu})_\tau^\sharp\not=\emptyset$
and $G_{\nu\tau}^\sharp=\emptyset$ or, equivalently,
$(\widetilde{K_\nu})_\varphi^\sharp\not=\emptyset$
and $K_{\nu\varphi}^\sharp=\emptyset$.
Remark also that the coadjoint orbit
${\mathcal O}^\varphi$ in
${\mathfrak k}_\nu^*$ passing through the point
$\varphi$ is isomorphic to the homogeneous spaces
$K_\nu/K_{\nu\varphi}$ and
$\widetilde{K_\nu}/(\widetilde{K_\nu})_\varphi$
simultaneously.

\begin{example}\label{ex.29}
Now we consider a connected and simply connected algebraic Lie group
$K=SU(3)$ and its representation
$\delta: SU(3)\to \operatorname{End}(gl(3,{\mathbb C}))$,
$\delta(k)(v)=kvk^t$, in the space $V$
of all complex matrices of order three (considered as a real
space). Here $k^t$ denotes the transpose of a matrix
$k\in SU(3)$. Using the nondegenerate 2-form
$\langle v_1,v_2\rangle=\operatorname{Re}\operatorname{Tr} v_1v_2$
on $V$ we identify the space $V$ with dual
$V^*$. Under this identification the dual representation
$\delta^*$ is given by
$\delta^*(k)(v)=k^tvk$. It is clear that for the covector
$\nu=E$, where
$E$ is the identity matrix, the isotropy group
$H=K_\nu$ is the group $SO(3)=SO(3,{\mathbb C})\cap SU(3)$.
Its universal covering group
$\widetilde H=\widetilde{K_\nu}$ is isomorphic to
$SU(2)$. But as we showed above (see Example~\ref{ex.25})
there is an element
$\varphi\in{\mathfrak h}^*={\mathfrak k}_\nu^*$ such that
$\widetilde H_\varphi^\sharp=
(\widetilde{K_\nu})_\varphi^\sharp\not=\emptyset$ while
$H_\varphi^\sharp=K_{\nu\varphi}^\sharp=\emptyset$.
Thus, as we proved above,
$(\widetilde{G_\nu})_\tau^\sharp\not=\emptyset$ while
$G_{\nu\tau}^\sharp=\emptyset$,
that is the condition of Proposition~\ref{pr.27} is not
sufficient.
\end{example}

\begin{remark}\label{re.30}
The Rawnsley's assertion~\cite[Corollary to Prop.2]{Ra}
claims that an arbitrary coadjoint orbit
${\mathcal O}^\sigma$ in the dual space
${\mathfrak g}^*$ of the semidirect product
${\mathfrak g}$ is integral if and only if the coadjoint
orbit ${\mathcal O}^\varphi\simeq K_\nu/K_{\nu\varphi}$ in
${\mathfrak k}_\nu^*$ is integral. From
Example~\ref{ex.29} it follows that in general this assertion is not
true. The gap in the proof of this assertion~\cite[Corollary
to Prop.2]{Ra} consists in an illegal using of Kostant's
theorem~\ref{th.24} (with the not necessary simply connected group
$H=K_\nu$).
\end{remark}

\end{document}